\documentclass[11pt]{elsarticle}

\usepackage[ddmmyyyy]{datetime}
\usepackage{geometry}
\usepackage{lineno,hyperref}
\usepackage{verbatim}
\nolinenumbers
\usepackage{float}
\usepackage{amsfonts}
\usepackage{mathrsfs}
\usepackage{amssymb}
\usepackage{amsmath}
\usepackage{amsthm}
\usepackage{appendix}
\usepackage{hyperref}
\usepackage{xcolor}
\usepackage[textfont=scriptsize,labelfont=bf]{caption}
\usepackage[textfont=scriptsize,labelfont=scriptsize]{subcaption}
\usepackage{multirow}
\usepackage{booktabs}
\usepackage{graphicx}
\usepackage{multirow,tabularx}
\usepackage[linesnumbered,lined,boxed,commentsnumbered,ruled]{algorithm2e}

\usepackage{TensorNeed}
\usepackage{TensorNeedMicros}

\makeatletter
\def\ps@pprintTitle{%
 \let\@oddhead\@empty
 \let\@evenhead\@empty
 \def\@oddfoot{}%
 \let\@evenfoot\@oddfoot}
\makeatother

\begin{document}

\begin{frontmatter}

\title{Fast Tensor Needlet Transforms for Tangent Vector Fields\\ on the Sphere\tnoteref{tdate}}
\tnotetext[tdate]{\today}
\author[addrltu]{Ming Li}
\ead{ming.li@latrobe.edu.au}

\author[addrltu]{Philip Broadbridge}
\ead{P.Broadbridge@latrobe.edu.au}

\author[addrltu]{Andriy Olenko}
\ead{a.olenko@latrobe.edu.au}

\author[addrltu,addrunsw]{Yu Guang Wang}
\ead{yuguang.wang@unsw.edu.au}
%
%
\address[addrltu]{Department of Mathematics and Statistics,
La Trobe University, Melbourne, VIC, 3086, Australia}
\address[addrunsw]{School of Mathematics and Statistics,
The University of New South Wales, Sydney, NSW, 2052, Australia}

\begin{abstract}
This paper constructs a semi-discrete tight frame of tensor needlets associated with a quadrature rule for tangent vector fields on the unit sphere $\sph{2}$ of $\Rd[3]$ --- tensor needlets. The proposed tight tensor needlets provide a multiscale representation of any square integrable tangent vector field on $\sph{2}$, which leads to a multiresolution analysis (MRA) for the field. From the MRA, we develop fast algorithms for tensor needlet transforms, including the decomposition and reconstruction of the needlet coefficients between levels, via a set of filter banks and scalar FFTs. The fast tensor needlet transforms have near linear computational cost proportional to $N\log \sqrt{N}$ for $N$ evaluation points or coefficients. Numerical examples for the simulated and real data demonstrate the efficiency of the proposed algorithm.
\end{abstract}


\end{frontmatter}

\section{Introduction}
Numerous processes studied in geosciences, planetary science and cosmology exhibit relatively small changes in their vertical (radial) direction compared to surface (tangential) directions. In many applications the vertical  component can be neglected and only tangential components are of main interest. \emph{Tangent (vector) fields} on the sphere appear in many real-world applications, such as geophysics, astrophysics, and environmental sciences, in which wind and oceanic currents, gravity, electric and magnetic fields are some of the most widely studied examples \cite{FrSc2009,GaLeSl2011, Giraldo2004, Holton1973,Watterson2001}. Tangent fields also provide an important tool for modelling evolutions of systems described by partial differential equation on the sphere, see e.g. \cite{Anh2018, GaLeSl2011}. In this paper, we develop a localized tight frame for tangent fields on $\sph{2}\subset\Rd[3]$ --- \emph{tensor needlets}. The tensor needlet system provides the multiscale decomposition of any square integrable spherical tangent field. From this, we develop an efficient computational strategy for the multiresolution analysis by tensor needlets, which we call fast tensor needlet transforms or FaTeNT. We then apply FaTeNT to model the global wind field. 

 It is known that each tangent field can be decomposed into two components: divergent-free and curl-free, and each part may provide some important physical insight in specific changes of  tangent fields. Compared with the available theory and learning tools for scalar fields on the sphere and other manifolds, modelling tangent fields has been less studied, which was the main motivation of this research.

Several techniques have been developed for tangent field approximation on the sphere. In the early 1980's, Wahba \cite{Wahba1982} used vector spherical harmonics (VSH) to construct the approximation of the tangent field. Freeden and Gervens \cite{Freeden1993} considered a similar approximation technique using vector spherical splines. They introduced positive definite kernels for fitting and decomposing a field using spherical basis functions (SBFs) \cite{Fengler2005,FrSc2009}.
The authors of \cite{Krishnamurti2018,Lynch1988} investigated the approximation of the stream function (divergence-free part) and velocity potential (curl-free part) of the field, which however relies on the computation of the divergence and vorticity of the field. In \cite{FuWr2009}, Fuselier and Wright studied the spherical basis function interpolation and approximation for  tangent fields on the sphere, providing both theoretical and simulation verifications on the stability and error estimations. However, a common shortcoming  for the interpolant-based approximation is that solving a linear system (to find the optimal coefficient vectors) can be computationally expensive when the set of scattered data points becomes large.

The multiresolution analysis is well developed to deal with sparse high-dimensional data in the Euclidean space $\R^{d}.$ It allows fast algorithm implementations for building the approximators.
Multiscale representation systems in $\R^{d}$ including wavelets, framelets, curvelets, shearlets, etc., have been well-developed and widely used, see e.g. \cite{BoKuZh2015,CaDeDoYi2006,Chui1992,CoDaVi1993,Daubechies1992,Mallat2009, Meyer1990}. The core of the classical framelet (and wavelet) construction relies on the extension principles such as unitary extension principle (UEP) \cite{RoSh1997}, oblique extension principle (OEP) and mixed extension principle (MEP) \cite{DaHaRoSh2003}. The extension principles associate framelet systems with filter banks, which enables fast algorithmic realizations for the framelet transforms and applications, see e.g. \cite{DaHaRoSh2003,HaZhZh2016,Mallat2009}. The fast algorithms include the filter bank decomposition and reconstruction of a representation system which uses the convolution and the fast Fourier transforms (FFTs) and achieve computational complexity in proportion to the size of the input data (up to a log factor). However, multiscale representation systems on the sphere and their corresponding fast algorithms are less studied. Wang and Zhuang \cite{WaZh2018} constructed a tight framelet system on a mainfold by means of orthogonal polynomials, localized kernels and affine systems. It was used to develop fast algorithmic realizations for the multi-level framelet filter bank transforms using discrete FFTs on  compact Riemannian manifolds. Similar approaches can be seen in Fischer, Mhaskar and Prestin in \cite{FiPr1997, MhPr2004}. Coifman, Maggioni, Mhaskar and Dong \cite{CoMa2006,Dong2017,MaMh2008,Mhaskar2010} considered more general cases, for which diffusion wavelets, diffusion polynomial frames and wavelet tight frames on manifolds and graphs were constructed. These frameworks have been well developed for scalar fields, but the case of vector fields on the sphere is little studied.

In this paper, we study multiresolution analysis for tangent fields on the sphere. Using the framework in \cite{WaZh2018} and fine properties of vector spherical harmonics, we first construct the tight frame of continuous tensor needlets and then semi-discrete ones. The approach is based on discretization of the continuous systems using a sequence of polynomial-exact quadrature rules on the sphere. We give the theoretical results on the equivalence conditions for a sequence of (continuous and semi-discrete) tensor needlets to be a sequence of tight frames on $\vLp{2}{2}$. These results offer tools to construct tensor needlet transforms and filter banks for a tight tensor needlet system. Then, we detail the multi-level tensor needlet transforms that include the decomposition and reconstruction. Namely, we decompose tensor needlet coefficients into a coarse scale approximation coefficient sequence (low-pass) together with multiple coarse scale detail coefficient sequences (high-pass) and reconstruct from the coarse scale approximation and details to fine scales. In particular, downsampling and upsampling operations are used in the decomposition and reconstruction, while (discrete) convolutions with filters in the filter bank are employed in both. Using the recently developed FFTs for vector spherical harmonics (FaVeST) in \cite{FaVest2019}, we can speed up the decomposition and reconstruction by implementing discrete Fourier transforms in the convolutional operations. It leads to fast tensor needlet transform (FaTeNT) algorithms. Numerical simulations, including simulated tangent fields with characteristics similar to those of atmospheric wind fields and real-world case study of climatological wind field, demonstrate the efficiency of our proposed algorithms.

The main contributions of the paper are:
\begin{itemize}
	\item The construction of tight continuous and semi-discrete tensor needlet systems with rigorous theoretical analysis of their tightness.
	\item The development of fast tensor needlet transforms with nearly linear computational complexity and low redundancy rate based on the recently proposed fast vector spherical harmonic transforms (FaVeST).
	\item  An extension of results in \cite{WaZh2018} to the case of tangent fields on the sphere.
	\item Detailed numerical studies that demonstrate that the developed fast algorithms work favourably on modelling tangent fields on the sphere, and also indicate a good potential on vectorial approximation via localized tight frames of tensor needlets.
\end{itemize}

The paper is organized as follows. In Section~\ref{sec2} we introduce the necessary background
and notation. The construction of tight frames of tensor needlets (for both continuous and discrete case) are given in detail
in Section~\ref{sec3}. Section~\ref{sec4} details the algorithmic implementations of fast tensor needlet transforms. In Section~\ref{sec5}, we show numerical studies to validate the proposed fast tensor needlet transforms.

\section{Preliminaries}\label{sec2}
Let $\C^{3}$ be the $3$-dimensional complex coordinate space.
In the following, the elements of $\C^{3}$ are column vectors. For $\PT{x}\in\C^{3}$, let the row vector $\PT{x}^{T}$ be the transpose of $\PT{x}$ and $\conj{\PT{x}}$ is the complex conjugate to $\PT{x}$. The inner product of two vectors $\PT{x}$ and $\PT{y}$ in $\C^{3}$ is
$\PT{x}\cdot\PT{y}=\sum_{i=1}^{3}x_i\conj{y_i}$, where $x_i,y_i$ are components of $\PT{x}$ and $\PT{y}$. The (Euclidean) $\ell_{2}$ norm of $\PT{x}$ is $|\PT{x}|=\sqrt{\PT{x}\cdot\PT{x}}$.
The tensor product for two vectors $\PT{x}$ and $\PT{y}$ in $\C^{3}$ is a matrix $\PT{x}\otimes\PT{y}$ such that
\begin{equation*}
	(\PT{x}\otimes\PT{y})_{i,j} = x_{i}\conj{y_{j}}.
\end{equation*}
	$\PT{x}\otimes\PT{y} = \PT{x}\conj{\PT{y}}^{T}$
is the matrix product of $\PT{x}$ with the transpose of $\PT{y}$. 

\subsection{Function spaces and vector spherical harmonics}
Let $\sph{2}$ be the unit sphere in $\Rd[3]$.
A $\C^{3}$-valued function (defined on $\sph{2}$) is called a \emph{vector field} on $\sph{2}$. For a vector field $\PT{f}$ on $\sph{2}$, the \emph{normal vector field} and \emph{tangent vector field} for $\PT{f}$ are
\begin{equation*}
	\PT{f}_{\rm nor}(\PT{x}):= (\PT{f}\cdot\PT{x})\PT{x},\quad
	\PT{f}_{\rm tan}(\PT{x}):= \PT{f} - \PT{f}_{\rm nor},
\end{equation*}
where $\PT{x}\in \sph{2}$.
A vector field $\PT{f}$ is called \emph{tangent (vector) field} if
$\PT{f}=\PT{f}_{\rm tan}$, i.e. $\PT{f}_{\rm nor}(\PT{x})\equiv 0$ on $\sph{2}$. In this paper, we study the tangent field.
Let $\vLp{2}{2}$ be the $L_2$ space of tangent fields on $\sph{2}$ with the inner product
\begin{equation*}
	\InnerL{\PT{f},\PT{g}} := \InnerL[\vLp{2}{2}]{\PT{f},\PT{g}} := \int_{\sph{2}} \PT{f}(\PT{x})\cdot\PT{g}(\PT{x})\IntDiff{x},\quad \PT{f},\PT{g}\in\vLp{2}{2}
\end{equation*}
and the induced finite $L_2$ norm $\normb{\PT{f}}{\vLp{2}{2}}:=\sqrt{\InnerL{\PT{f},\PT{f}}}$. For a tangent field $\PT{f}=(f_1,f_2,f_3)^{T}\in \vLp{2}{2}$, the integral $\int_{\sph{2}}\PT{f}(\PT{x})\IntDiff{x}$ of $\PT{f}$ on $\sph{2}$ denotes the vector of componentwise integrals $(\int_{\sph{2}}f_1(\PT{x})\IntDiff{x},\int_{\sph{2}}f_2(\PT{x})\IntDiff{x},\int_{\sph{2}}f_3(\PT{x})\IntDiff{x})^{T}$.
A $\C^{3\times 3}$-valued \emph{tensor field} $\mathbf{u}$ on $\sph{2}$ is a $\C^{3\times 3}$-valued function on $\sph{2}$. Suppose that each column of $\C^{3\times 3}$-valued tensor field $\PT{u}$ is a tangent field in $\vLp{2}{2}$, that is, $\mathbf{u}=\left(\PT{u}_1,\PT{u}_2,\PT{u}_3\right)$ and $\PT{u}_i\in \vLp{2}{2}$, $i=1,2,3$. For $\PT{f}\in\vLp{2}{2}$, we define the ``inner product'' between $\PT{f}$ and $\mathbf{u}$ by
\begin{equation*}
	\InnerL{\PT{f},\mathbf{u}}:=
	\begin{pmatrix}
		\InnerL{\PT{f},\PT{u}_1}, & \InnerL{\PT{f},\PT{u}_2}, & \InnerL{\PT{f},\PT{u}_3}
	\end{pmatrix}.
\end{equation*}
For $\PT{f},\PT{g},\PT{h}\in \vLp{2}{2}$ and $\PT{y}\in\sph{2}$,
\begin{equation*}
	\InnerL{\PT{f},\PT{g}\otimes\PT{h}(\PT{y})} = \InnerL{\PT{f},\PT{g}}\PT{h}(\PT{y})^{T},\quad
	\InnerL{\PT{h}(\PT{y})\otimes\PT{g},\PT{f}} = \PT{h}(\PT{y})\InnerL{\conj{\PT{g}},\PT{f}}=\InnerL{\conj{\PT{f}},\PT{g}}\PT{h}(\PT{y}),
\end{equation*}
where $\PT{h}(\PT{y})^{T}:=(\PT{h}(\PT{y}))^{T}$ for simplicity. We will use this notation in the paper if no confusion arises.

Let $\LBo$ be the Laplace-Beltrami operator on $\sph{2}$.  Complex-valued spherical harmonics $\shY$, $\ell\in\Nz, m=-\ell,\dots,\ell$, are eigenfunctions of $\LBo$ with eigenvalues $\eigvm:=\ell(\ell+1)$ satisfying
\begin{equation*}
	\LBo \:\shY = \eigvm\: \shY.
\end{equation*}
Let $\nabla$ be the gradient on $\Rd[3]$. The \emph{surface-gradient} on $\sph{2}$ is defined as $\sfgrad:=P_{x}\nabla$ with the matrix $P_{\PT{x}} :=\imat-\PT{x}\PT{x}^{T}$, where $\imat$ is the identity matrix. The \emph{surface-curl} on $\sph{2}$ is $\sfcurl =Q_{x}\nabla$ with the matrix
\begin{equation*}
	Q_{\PT{x}} :=
	\begin{pmatrix}
		0 & -x_3 & x_2\\
		x_3 & 0 & -x_1\\
		-x_2 & x_1 & 0
	\end{pmatrix}.
\end{equation*}
Then, it holds $\sfgrad\cdot \sfgrad = \sfcurl\cdot \sfcurl = -\LBo$.
The \emph{divergence-free vector spherical harmonics} are
\begin{equation*}
	\dsh = \sfcurl \shY/\sqrt{\eigvm}, \quad \ell\ge1.
\end{equation*}
The \emph{curl-free vector spherical harmonics} are
\begin{equation*}
	\csh = \sfgrad \shY/\sqrt{\eigvm}, \quad \ell\ge1.
\end{equation*}
The set $\{(\dsh,\csh) \setsep \ell\ge1,m=-\ell,\dots,\ell\}$ forms an orthonomal basis of $\vLp{2}{2}$, see e.g. \cite{FrSc2009, FuWr2009}.
Let $\dfco{\PT{f}}:=\InnerL{\PT{f},\dsh}$ and $\cfco{\PT{f}}:=\InnerL{\PT{f},\csh}$ be the (divergence-free and curl-free) Fourier coefficients for $\PT{f}\in\vLp{2}{2}$.
For $\PT{f}\in\vLp{2}{2}$, it holds
\begin{equation*}
	\PT{f} = \sum_{\ell=1}^{\infty}\sum_{m=-\ell}^{\ell} \left(\dfco{\PT{f}}\: \dsh + \cfco{\PT{f}}\: \csh \right)
\end{equation*}
in $\vLp{2}{2}$ space.
For $\ell\ge1$, the divergence-free and curl-free \emph{Legendre rank-$2$ kernels} of degree $\ell$ are the tensor fields
\begin{equation*}
	\dlg(\PT{x},\PT{y}) = \frac{2\ell+1}{4\pi\eigvm}(\sfgrad)_{\PT{x}}\otimes(\sfgrad)_{\PT{y}}\Legen(\PT{x}\cdot\PT{y}),\quad
	\clg(\PT{x},\PT{y}) = \frac{2\ell+1}{4\pi\eigvm}{\sfcurl}_{\PT{x}}\otimes{\sfcurl}_{\PT{y}}\Legen(\PT{x}\cdot\PT{y}),\quad \PT{x},\PT{y}\in\sph{2}.
\end{equation*}
For vector spherical harmonics and Legendre rank-$2$ kernels, the following addition theorem holds
\begin{equation*}
	\sum_{m=-\ell}^{\ell}\dsh(\PT{x})\otimes\dsh(\PT{y}) = \dlg(
	\PT{x},\PT{y}),\quad
	\sum_{m=-\ell}^{\ell}\csh(\PT{x})\otimes\csh(\PT{y}) = \clg(
	\PT{x},\PT{y})
\end{equation*}
for $\ell\ge 1$ and $\PT{x},\PT{y}\in\sph{2}$, see \cite[Theorem~5.3]{FrSc2009}.

\section{Tight Frame of Tensor Needlets}\label{sec3}
In this section, we construct a localized tight frame on tangent fields $\vrf$ that are defined on $\sph{2}$ and taking values in $\Rd[d]$. Here we inherit the notation and concepts of \cite{HaZh2015,LeSlWaWo2017, WaLeSlWo2017,WaZh2018}.

Let $\Lpr{2}$ be the space of complex-valued square-integrable functions on $\Rone$, and let $\ellp{2}$ be a space of square-summable sequences on $\Z$. The Fourier transform of $\phi\in\Lpr{2}$ is $\FT{\phi}(\xi):=\int_{\Rone}\phi(x)e^{-2\pi \imu x\xi}\IntD{\xi}$, $\xi\in\Rone$. The Fourier series for $h\in \ellp{2}$ is $\FS{h}:=\sum_{k\in\Z}h_k e^{-2\pi \imu k\xi}$, $\xi\in\Rone$.
Let \begin{equation}\label{eq:scal.sys}
	\digamma:=\{\scala,\scalb^{1},\dots,\scalb^{r}\}
\end{equation}
be a set of functions in $\Lpr{2}$. We call $\digamma$ a set of \emph{generating functions} if it is associated with a \emph{filter bank}
\begin{equation}\label{eq:filtbk}
	\filtbk := \{\maska,\maskb[1],\dots,\maskb[r]\}\subset \ellp{2}
\end{equation}
satisfying
\begin{equation}\label{eq:scaling.mask}
	\FT{\scala}(2\xi)=\FT{\scala}(\xi)\FS{\maska}(\xi),\quad
	\FT{\scalb^{n}}(2\xi) = \FT{\scala}(\xi)\FS{\maskb}(\xi),\quad n=1,\dots,r.
\end{equation}

\begin{definition}\label{defn:cfr}
	For a set of scaling functions $\digamma$ given by \eqref{eq:scal.sys}, the ($\C^{3\times 3}$-valued) continuous tensor needlets for tangent fields on $\sph{2}$ are
	\begin{align*}
	\cfra(\PT{x}) &:= \sum_{\ell=1}^{\infty} \FS{\scala}\left(\frac{\ell}{2^{j}}\right)\sum_{m=-\ell}^{\ell}\Bigl(\dsh(\PT{x})\otimes\dsh(\PT{y}) + \csh(\PT{x})\otimes\csh(\PT{y})\Bigr)\\
	&=\sum_{\ell=1}^{\infty} \FS{\scala}\left(\frac{\ell}{2^{j}}\right)\Bigl(\dlg(\PT{x},\PT{y}) + \clg(\PT{x},\PT{y})\Bigr),\\
	\cfrb{n}(\PT{x}) &:= \sum_{\ell=1}^{\infty} \FS{\scalb^{n}}\left(\frac{\ell}{2^{j}}\right)\sum_{m=-\ell}^{\ell}\Bigl(\dsh(\PT{x})\otimes\dsh(\PT{y}) + \csh(\PT{x})\otimes\csh(\PT{y})\Bigr)\\
	&=\sum_{\ell=1}^{\infty} \FS{\scalb^{n}}\left(\frac{\ell}{2^{j}}\right)\Bigl(\dlg(\PT{x},\PT{y}) + \clg(\PT{x},\PT{y})\Bigr), \quad n=1,\dots,r.
	\end{align*}
\end{definition}

For $J=0,1,\dots$, the set of tensor needlets
\begin{equation}\label{eq:cfrsys}
	\cfrsys:=\{\cfra[J,\PT{y}];\cfrb{1},\dots,\cfrb{r}: j\geq J, \PT{y}\in\sph{2}\}
\end{equation}
is called the \emph{continuous needlet system} starting from scale level $J$ with the filter bank $\filtbk$.
Each continuous needlet is a tensor field on $\sph{2}$.
We call the continuous tensor needlet system $\cfrsys$ a tight needlet system for $\vLp{2}{2}$ if $\cfrsys\subset\vLp{2}{2}$ and if for any $\PT{f} \in \vLp{2}{2}$,
\begin{equation}\label{eq:tight.cfrsys}
	\PT{f} = \int_{\sph{2}}\cfra[J,\PT{y}]\InnerL{\PT{f},\cfra[J,\PT{y}]}^{T}\IntDiff{y}
	 + \sum_{j=J}^{\infty}\sum_{n=1}^{r}\int_{\sph{2}}\cfrb{n}\InnerL{\PT{f},\cfrb{n}}^{T}\IntDiff{y},
\end{equation}
in $L_2$ sense or equivalently,
\begin{equation*}\label{eq:tight.cfrsys.coeff}
	\norm{\PT{f}}{\vLp{2}{2}}^{2}=\int_{\sph{2}}\left|\InnerL{\PT{f},\cfra[J,\PT{y}]}\right|^{2}\IntDiff{y}
	+ \sum_{j=J}^{\infty}\sum_{n=1}^{r}\int_{\sph{2}}\left|\InnerL{\PT{f},\cfrb{n}}\right|^{2}\IntDiff{y}.
\end{equation*}
The elements of the tight continuous tensor needlet system $\cfrsys$ will be called \emph{continuous tight tensor needlets} for $\vLp{2}{2}$.

\begin{theorem}\label{thm:equiv.tight.cfrsys}
	Let $J_0\in \Z$ be an integer, $\cfrsys(\filtbk)$, $J\geq J_0$, given in \eqref{eq:cfrsys}, be a sequence of continuous tensor needlet systems whose continuous tensor needlets are given by Definition~\ref{defn:cfr} with filter bank $\filtbk$ given by \eqref{eq:filtbk} and scaling functions satisfying \eqref{eq:scaling.mask}. Then the following statements are equivalent.\\
(i) The continuous needlet systems $\cfrsys(\filtbk)$ are tight frames for $\vLp{2}{2}$ for $J\geq J_0$, i.e. \eqref{eq:tight.cfrsys} holds for all $J\geq J_0$.\\
(ii) For each $\PT{f}\in \vLp{2}{2}$, the following identities hold.
\begin{align}
	&\lim_{j\to\infty} \normB{\int_{\sph{2}}\cfra \InnerL{\PT{f},\cfra}^{T}\IntDiff{y}-\PT{f}}{\vLp{2}{2}} = 0,\label{eq:cfrpra.f.L2converg}\\
	& \hspace{0.2cm}\mbox{and for all $j\geq J_0$},\notag\\
	& \int_{\sph{2}}\cfra[j+1,\PT{y}]\InnerL{\PT{f},\cfra[j+1,\PT{y}]}^{T}\IntDiff{y}
	= \int_{\sph{2}}\cfra \InnerL{\PT{f},\cfra}^{T}\IntDiff{y}
	+ \sum_{n=1}^{r}\int_{\sph{2}}\cfrb{n}\InnerL{\PT{f},\cfrb{n}}^{T}\IntDiff{y}.\label{eq:cfrpraj.j+1}
\end{align}
(iii) For each $\PT{f}\in \vLp{2}{2}$, the following identities hold.
\begin{align*}
	&\lim_{j\to\infty} \int_{\sph{2}}\left|\InnerL{\PT{f},\cfra}\right|^{2}\IntDiff{y}= \norm{\PT{f}}{\vLp{2}{2}}^{2},\\
	& \hspace{-0.5cm}\mbox{and for all $j\geq J_0$},\\
	& \int_{\sph{2}}\left|\InnerL{\PT{f},\cfra[j+1,\PT{y}]}\right|^{2}\IntDiff{y}
	= \int_{\sph{2}}\left|\InnerL{\PT{f},\cfra}\right|^{2}\IntDiff{y}
	+ \sum_{n=1}^{r}\int_{\sph{2}}\left|\InnerL{\PT{f},\cfrb{n}}\right|^{2}\IntDiff{y}.
\end{align*}
(iv) The scaling functions in $\digamma$ satisfy
\begin{align}
	&\lim_{j\to\infty}\left|\FT{\scala}\left(\frac{\eigvm}{2^{j}}\right)\right| =1, \quad \ell \geq1,\label{eq:scala.lim}\\
	& \left|\FT{\scala}\left(\frac{\eigvm}{2^{j+1}}\right)\right|^{2}
	= \left|\FT{\scala}\left(\frac{\eigvm}{2^{j}}\right)\right|^{2}
	+ \sum_{n=1}^{r}\left|\FT{\scalb^{n}}\left(\frac{\eigvm}{2^{j}}\right)\right|^{2},\quad \ell\geq \ell, \; j\geq J_0.\label{eq:scal.j.j+1}
\end{align}
(v) The refinable function $\scala$ satisfies \eqref{eq:scala.lim} and for all $\ell$ satisfying $\FS{\maska}\left(\frac{\eigvm}{2^{j}}\right)\neq0$ and $j\geq J_0+1$ the filters in the filter bank $\filtbk$ satisfy
\begin{equation*}
	\left|\FS{\maska}\left(\frac{\eigvm}{2^{j}}\right)\right|^2
	+ \sum_{n=1}^{r}\left|\FS{\maskb}\left(\frac{\eigvm}{2^{j}}\right)\right|^2=1.
\end{equation*}
\end{theorem}
\begin{proof}
	We first show that ``(i) $\Longleftrightarrow$ (ii)''. For $j\in\N$, let
	\begin{equation*}
		\cfrpra{j}\PT{f} := \int_{\sph{2}}\cfra[j,\PT{y}]\InnerL{\PT{f},\cfra[j,\PT{y}]}^{T}\IntDiff{y},\quad
		\cfrprb{j}\PT{f} := \int_{\sph{2}}\cfrb{n}\InnerL{\PT{f},\cfrb{n}}^{T}\IntDiff{y},\quad \PT{f}\in\vLp{2}{2}.
	\end{equation*}
	
	``(i) $\Longleftarrow$ (ii)''.
	If $\cfrsys(\filtbk)$ is a tight frame, then for all $\PT{f}\in\vLp{2}{2}$ and all $J\ge J_{0}$
	\begin{equation*}
		\PT{f} = \cfrpra{J}\PT{f} + \sum_{j=J}^{\infty}\sum_{n=1}^{r}\cfrprb{j}\PT{f} = \cfrpra{J+1}\PT{f} + \sum_{j=J+1}^{\infty}\sum_{n=1}^{r}\cfrprb{j}\PT{f}
	\end{equation*}
  in $L_2$ sense. Thus, for $J\ge J_{0}$,
	\begin{equation}\label{eq:cfrpraJ.J+1}
		\cfrpra{J+1}\PT{f} = \cfrpra{J}\PT{f} + \sum_{n=1}^{r}\cfrprb{J}\PT{f}.
	\end{equation}
	This gives \eqref{eq:cfrpraj.j+1}. Recursively using \eqref{eq:cfrpraJ.J+1} we obtain
	\begin{equation}\label{eq:cfrpraJ.m+1}
		\cfrpra{m+1}\PT{f} = \cfrpra{J}\PT{f} + \sum_{j=J}^{m}\sum_{n=1}^{r}\cfrprb{j}\PT{f}
	\end{equation}
	for all $m\ge J$ and $J\ge J_{0}$. If $m\to \infty$ in \eqref{eq:cfrpraJ.m+1}, we then obtain
	\begin{equation*}
		\lim_{m\to\infty} \cfrpra{m+1}\PT{f} = \cfrpra{J}\PT{f} + \sum_{j=J}^{\infty}\sum_{n=1}^{r}\cfrprb{j}\PT{f}=\PT{f},\quad \PT{f}\in\vLp{2}{2}.
	\end{equation*}
	This gives \eqref{eq:cfrpra.f.L2converg}.
	
	``(i)$\Longrightarrow$(ii)''. Using \eqref{eq:cfrpraj.j+1} recursively, we obtain \eqref{eq:cfrpraJ.m+1}. If $m$ tends to infinity, with the convergence of \eqref{eq:cfrpra.f.L2converg}, we obtain the tightness of $\cfrsys$ for $J\ge J_0$.
	
	We now prove ``(ii)$\Longleftrightarrow$(iv)''. By Definition~\ref{defn:cfr}, for $n=1,\dots,r$, $j=0,1,\dots$ and $\PT{y}\in\sph{2}$,
	\begin{align*}
		\InnerL{\PT{f},\cfra}^{T} &= \sum_{\ell=0}^{\infty}\conj{\FT{\scala}\left(\frac{\eigvm}{2^{j}}\right)}\sum_{m=-\ell}^{\ell}\left(\dfco{\PT{f}}\dsh(\PT{y})+\cfco{\PT{f}}\csh(\PT{y})\right)\\
			\InnerL{\PT{f},\cfrb{n}}^{T} &= \sum_{\ell=0}^{\infty}\conj{\FT{\scalb^{n}}\left(\frac{\eigvm}{2^{j}}\right)}\sum_{m=-\ell}^{\ell}\left(\dfco{\PT{f}}\dsh(\PT{y})+\cfco{\PT{f}} \csh(\PT{y})\right).
	\end{align*}
	Then,
	\begin{align}
		\cfrpra{j}\PT{f} &= \int_{\sph{2}}\sum_{\ell=1}^{\infty}\FT{\scala}\left(\frac{\eigvm}{2^{j}}\right)\sum_{m=-\ell}^{\ell}\bigl(\dsh\otimes\dsh(\PT{y})+\csh\otimes\csh(\PT{y})\bigr)\notag\\
		&\hspace{2cm}\times \sum_{\ell=1}^{\infty}\conj{\FT{\scala}\left(\frac{\eigvm}{2^{j}}\right)}\sum_{m=-\ell}^{\ell}\left(\dfco{\PT{f}}\dsh(\PT{y})+\cfco{\PT{f}}\csh(\PT{y})\right)\IntDiff{y}\notag\\
		&=\sum_{\ell=1}^{\infty}\sum_{\ell'=1}^{\infty}\conj{\FT{\scala}\left(\frac{\eigvm}{2^{j}}\right)}\FT{\scala}\left(\frac{\eigvm[\ell']}{2^{j}}\right)\sum_{m=-\ell}^{\ell}\sum_{m'=-\ell'}^{\ell'}\left(\dsh\InnerL{\dsh[\ell' m'],\dsh}\dfco[\ell' m']{\PT{f}}
		+\dsh\InnerL{\csh[\ell' m'],\dsh}\cfco[\ell' m']{\PT{f}}\right.\notag\\
		&\hspace{2cm}\left. + \csh\InnerL{\dsh[\ell' m'],\csh}\dfco[\ell' m']{\PT{f}} + \csh\InnerL{\csh[\ell' m'],\csh}\cfco[\ell' m']{\PT{f}}\right)\notag\\
		&=\sum_{\ell=1}^{\infty}\left|\FT{\scala}\left(\frac{\eigvm}{2^{j}}\right)\right|^{2}\sum_{m=-\ell}^{\ell}\left(\dfco{\PT{f}}\dsh+\cfco{\PT{f}}\csh\right).\label{eq:cfraprj.f}
	\end{align}
Similarly,
\begin{equation*}
	\cfrprb{j}\PT{f}=\sum_{\ell=1}^{\infty}\left|\FT{\scalb^{n}}\left(\frac{\eigvm}{2^{j}}\right)\right|^{2}\sum_{m=-\ell}^{\ell}\left(\dfco{\PT{f}}\dsh+\cfco{\PT{f}}\csh\right).
\end{equation*}

Thus,
\begin{align}
	\dfco{\bigl(\cfrpra{j}\PT{f}\bigr)} &= \left|\FT{\scala}\left(\frac{\eigvm}{2^{j}}\right)\right|^{2}\dfco{\PT{f}},\quad \cfco{\bigl(\cfrpra{j}\PT{f}\bigr)} = \left|\FT{\scala}\left(\frac{\eigvm}{2^{j}}\right)\right|^{2}\cfco{\PT{f}},\label{eq:fco.cfaprj}\\
		\dfco{\bigl(\cfrprb{j}\PT{f}\bigr)} &= \left|\FT{\scalb^{n}}\left(\frac{\eigvm}{2^{j}}\right)\right|^{2}\dfco{\PT{f}},\quad \cfco{\bigl(\cfrprb{j}\PT{f}\bigr)} = \left|\FT{\scalb^{n}}\left(\frac{\eigvm}{2^{j}}\right)\right|^{2}\cfco{\PT{f}}.\label{eq:fco.cfbprj}
\end{align}
This means that \eqref{eq:cfrpraj.j+1} is equivalent to \eqref{eq:scal.j.j+1}.
Also, \eqref{eq:fco.cfaprj} and \eqref{eq:fco.cfbprj} give
\begin{equation*}
	\normb{\cfrpra{j}\PT{f}-\PT{f}}{\vLp{2}{2}}^{2}=\sum_{\ell=0}^{\infty}\left(\left|\FT{\scala}\left(\frac{\eigvm}{2^{j}}\right)\right|^{2}-1\right)^{2}\sum_{m=-\ell}^{\ell}\left(\left|\dfco{\PT{f}}\right|^{2} + \left|\cfco{\PT{f}}\right|^{2}\right)
\end{equation*}
and we obtain that \eqref{eq:cfrpra.f.L2converg} is equivalent to \eqref{eq:scala.lim}.
	
The equivalence between (iii) and (iv) is well-known, see e.g. \cite{WaZh2018}. The equivalence between (iv) and (v) follows from \eqref{eq:scaling.mask}.
\end{proof}

A \emph{quadrature rule} on $\sph{2}$ is a set of $N$ pairs of weights and points on $\sph{2}$
\begin{equation*}
	\QN[N]=\{(\wN[i],\pN[i])\setsep \wN[i]\in\Rone, \pN[i]\in\sph{2}, i=1,\dots,N\}.
\end{equation*}
For $L\geq0$, let $\Pi_{L}:={\rm span}\{Y_{\ell,m}:\ell=0,\dots, L,\; m=-\ell,\dots,\ell\}$ be a polynomial space of degree $L$. Elements of $\Pi_L$ are called polynomials of degree $L$. The quadrature rule $\QN[N]$ is called exact for polynomials of degree $L$ if for any $ P\in\Pi_L$
\begin{equation*}
	\int_{\sph{2}}P(\PT{x})\IntDiff{x} = \sum_{i=1}^{N}\wN[i]P(\pN[i]).
\end{equation*}

To obtain semi-discrete needlets, we discretize the integrals in \eqref{eq:tight.cfrsys} by quadrature rules for different scales.

\begin{definition}\label{defn:fr}
	Let  a set of scaling functions $\digamma$ be given by \eqref{eq:scal.sys} and a sequence of quadrature rules $\QN=\{(\wN,\pN)\}_{k=1}^{N_j}$ be exact for polynomials of degree $2^{2j}$. Then, for $j=0,1,\dots,$ and $k=1,\dots,N_{j}$, the ($\C^{d\times d}$-valued semi-discrete) tensor needlets for tangent fields on $\sph{2}$ are given by
	\begin{align*}
	\fra(\PT{x}) &:= \sqrt{\wN}\:\cfra[j,\pN](\PT{x})\\
	&=\sqrt{\wN}\sum_{\ell=1}^{\infty} \FS{\scala}\left(\frac{\eigvm}{2^{j}}\right)\sum_{m=-\ell}^{\ell}\left(\dsh(\PT{x})\otimes\dsh(\pN) + \csh(\PT{x})\otimes\csh(\pN)\right)\\
	&=\sqrt{\wN}\sum_{\ell=1}^{\infty} \FS{\scala}\left(\frac{\eigvm}{2^{j}}\right)\left(\dlg(\PT{x},\pN) + \clg(\PT{x},\pN)\right),\\
	&\hspace{-1.8cm}\mbox{and for $n=1,\dots,r$, $j=0,1,\dots,$ and $k'=1,\dots,N_{j+1}$},\\
	\frb[j,k']{n}(\PT{x}) &:= \sqrt{\wN[j+1,k']}\:\cfrb[j,{\pN[j+1,k']}]{n}(\PT{x})\\
	&=\sqrt{\wN[j+1,k']}\sum_{\ell=1}^{\infty} \FS{\scalb^{n}}\left(\frac{\eigvm}{2^{j}}\right)\sum_{m=-\ell}^{\ell}\left(\dsh(\PT{x})\otimes\dsh(\pN[j+1,k']) + \csh(\PT{x})\otimes\csh(\pN[j+1,k'])\right)\\
	&=\sqrt{\wN[j+1,k']}\sum_{\ell=1}^{\infty} \FS{\scalb^{n}}\left(\frac{\eigvm}{2^{j}}\right)\left(\dlg(\PT{x},\pN[j+1,k']) + \clg(\PT{x},\pN[j+1,k'])\right).
	\end{align*}
\end{definition}

For $J=0,1,\dots$, the set of needlets
\begin{equation}\label{eq:frsys}
	\frsys:=\frsys(\filtbk;\{\QN\}_{j\geq J}):=\{\fra[J,k];\frb[j,k']{1},\dots,\frb[j,k']{r}: j\geq J,\: k=1,\dots,N_{J},\: k'=1,\dots,N_{j+1}\}
\end{equation}
is called the \emph{(semi-discrete) tensor needlet system} starting from the scale level $J$ with the filter bank $\filtbk$.
The tensor needlet system $\frsys$ is called a tight needlet system for $\Lp{2}{2}$ if $\frsys\subset\Lp{2}{2}$ and for any $\PT{f} \in \vLp{2}{2}$
\begin{equation}\label{eq:tight.frsys}
	\PT{f} = \sum_{k=1}^{N_J}\fra[J,k]\InnerL{\PT{f},\fra[J,k]}^{T}
	 + \sum_{j=J}^{\infty}\sum_{k=1}^{N_{j+1}}\sum_{n=1}^{r}\frb{n}\InnerL{\PT{f},\frb{n}}^{T}
\end{equation}
in $L_2$ sense or equivalently,
\begin{equation*}\label{eq:tight.frsys.coeff}
	\norm{\PT{f}}{\vLp{2}{2}}^{2}=\sum_{k=1}^{N_J}\left|\InnerL{\PT{f},\fra[J,k]}\right|^{2}
	 + \sum_{j=J}^{\infty}\sum_{k=1}^{N_{j+1}}\sum_{n=1}^{r}\left|\InnerL{\PT{f},\frb{n}}\right|^{2}.
\end{equation*}
The elements of the tight tensor needlet system $\frsys$ are said to be \emph{(semi-discrete) tight tensor needlets} for $\vLp{2}{2}$.

\begin{theorem}\label{thm:equiv.tight.frsys}
	Let $\frsys(\filtbk;\{\QN\}_{j\geq J})$, $J\geq J_0$, given in \eqref{eq:frsys} be a sequence of tensor needlet systems with elements given by Definition~\ref{defn:cfr} and the filter bank $\filtbk$ given by \eqref{eq:filtbk}, scaling functions satisfying \eqref{eq:scaling.mask}. Let a sequence of quadrature rules $\QN=\{(\wN,\pN)\}_{k=1}^{N_j}$ be exact for degree $2^{j+1}$, $j=0,1,\dots$. Then, the following statements are equivalent.\\
(i) The continuous needlet systems $\frsys$ are tight frames for $\vLp{2}{2}$ for $J\geq J_0$, i.e. \eqref{eq:tight.frsys} holds for all $J\geq J_0$.\\
(ii) For each $\PT{f}\in \vLp{2}{2}$, the following identities hold.
\begin{align*}
	&\lim_{j\to\infty} \normB{\sum_{k=1}^{N_j}\fra \InnerL{\PT{f},\fra}^{T}-\PT{f}}{\vLp{2}{2}} = 0\\
	& \hspace{-0.5cm}\mbox{and for $j\geq J_0$}\\
	& \sum_{k=1}^{N_{j+1}}\fra[j+1,k]\InnerL{\PT{f},\fra[j+1,k]}^{T}
	= \sum_{k=1}^{N_j}\fra\InnerL{\PT{f},\fra}^{T}
	+ \sum_{k=1}^{N_{j+1}}\sum_{n=1}^{r}\frb{n}\InnerL{\PT{f},\frb{n}}^{T}.
\end{align*}
(iii) For each $\PT{f}\in \vLp{2}{2}$, the following identities hold.
\begin{align*}
	&\lim_{j\to\infty} \sum_{k=1}^{N_{j}}\left|\InnerL{\PT{f},\fra}\right|^{2}= \norm{\PT{f}}{\vLp{2}{2}}^{2}\\
	& \hspace{-0.5cm}\mbox{and for $j\geq J_0$}\\
	& \sum_{k=1}^{N_{j+1}}\left|\InnerL{\PT{f},\fra[j+1,k]}\right|^{2}
	= \sum_{k=1}^{N_{j}}\left|\InnerL{\PT{f},\fra}\right|^{2}
	+ \sum_{k=1}^{N_{j+1}}\sum_{n=1}^{r}\left|\InnerL{\PT{f},\frb{n}}\right|^{2}.
\end{align*}
(iv) The scaling functions in $\digamma$ satisfy
\begin{align}
	&\lim_{j\to\infty}\left|\FT{\scala}\left(\frac{\eigvm}{2^{j}}\right)\right| =1, \quad \ell \geq1,\label{eq:scala.lim.semidis}\\
	& \left|\FT{\scala}\left(\frac{\eigvm}{2^{j+1}}\right)\right|^{2}
	= \left|\FT{\scala}\left(\frac{\eigvm}{2^{j}}\right)\right|^{2}
	+ \sum_{n=1}^{r}\left|\FT{\scalb^{n}}\left(\frac{\eigvm}{2^{j}}\right)\right|^{2},\quad \ell\geq \ell, \; j\geq J_0.\notag \label{thmeq:cfr.scal.j.j1}
\end{align}
(v) The refinable function $\scala$ satisfies \eqref{eq:scala.lim.semidis} and for all $\ell$ satisfying $\FS{\maska}\left(\frac{\eigvm}{2^{j}}\right)\neq0$ and $j\geq J_0+1$ the filters in the filter bank $\filtbk$ satisfy
\begin{equation}\label{eq:filters_union}
	\left|\FS{\maska}\left(\frac{\eigvm}{2^{j}}\right)\right|^2
	+ \sum_{n=1}^{r}\left|\FS{\maskb}\left(\frac{\eigvm}{2^{j}}\right)\right|^2=1.
\end{equation}
\end{theorem}
\begin{proof}
	Since the components of $\dsh$ and $\csh$ are scalar polynomials on $\sph{2}$ with degrees not exceeding $\ell$, $\conj{\dsh(\PT{x})}^{T}\dsh[\ell' m'](\PT{x})$, $\conj{\dsh(\PT{x})}^{T}\csh[\ell' m'](\PT{x})$, $\conj{\csh(\PT{x})}\csh[\ell' m'](\PT{x})$ and their complex conjugates are scalar polynomials of $\PT{x}$ with degrees not exceeding $\ell+\ell'$. For $j\ge0$, as $\{(\wN,\pN)\}_{k=1}^{N_{j}}$ is a quadrature rule exact for degree $2^{j+1}$, by the orthogonality of $\dsh$ and $\csh$, for $\ell,\ell\ge1$, $m,m'=-\ell,\dots,\ell$, one gets
	\begin{align}
		&\sum_{k=1}^{N_{j}}\wN\:\conj{\dsh(\pN)}^{T}\dsh[\ell' m'](\pN)
		= \sum_{k=1}^{N_{j}}\wN\:\conj{\csh(\pN)}^{T}\csh[\ell' m'](\pN)=\delta_{\ell \ell'}\delta_{m m'},\label{eq:vsh.InnerD}\\
		&\sum_{k=1}^{N_{j}}\wN\:\conj{\dsh(\pN)}^{T}\csh[\ell' m'](\pN)
		=\sum_{k=1}^{N_{j}}\wN\:\conj{\csh(\pN)}^{T}\dsh[\ell' m'](\pN)=0.\label{eq:vsh.InnerD2}
	\end{align}
	By Definition~\ref{defn:fr}, for $n=1,\dots,r$ and $k=1,\dots,N_j$, $j=0,1,\dots$, it follows
	\begin{align}
		\InnerL{\PT{f},\fra}^{T} &= \sum_{\ell=1}^{\infty}\sqrt{\wN}\:\conj{\FS{\scala}\left(\frac{\eigvm}{2^{j}}\right)} \sum_{m=-\ell}^{\ell}\left(\dfco{\PT{f}}\dsh(\pN)+\cfco{\PT{f}}\csh(\pN)\right) \label{eq:fr.coeff1}\\
	    \InnerL{\PT{f},\frb{n}}^{T} &= \sum_{\ell=1}^{\infty}\sqrt{\wN[j+1,k]}\:\conj{\FS{\scalb^{n}}\left(\frac{\eigvm}{2^{j}}\right)} \sum_{m=-\ell}^{\ell}\left(\dfco{\PT{f}}\dsh(\pN)+\cfco{\PT{f}}\csh(\pN)\right)\label{eq:fr.coeff2}.
	\end{align}
	Then,
	\begin{align*}
		&\sum_{k=1}^{N_{j}}\fra\InnerL{\PT{f},\fra}^{T}\\
		&\quad= \sum_{k=1}^{N_j} \sum_{\ell=1}^{\infty}\wN\:\FS{\scala}\left(\frac{\eigvm}{2^{j}}\right)
		 \sum_{m=-\ell}^{\ell} \left(\dsh\otimes \dsh(\pN) + \csh\otimes \csh(\pN)\right)\\
		&\hspace{2.2cm}\times \sum_{\ell=1}^{\infty}\conj{\FT{\scala}\left(\frac{\eigvm}{2^{j}}\right)}\sum_{m=-\ell}^{\ell}\left(\dfco{\PT{f}}\dsh(\pN)+\cfco{\PT{f}}\csh(\pN)\right)\\
		&\quad = \sum_{k=1}^{N_j}\sum_{\ell=1}^{\infty}\sum_{\ell'=1}^{\infty}\wN\:\FS{\scala}\left(\frac{\eigvm}{2^{j}}\right)
		\conj{\FS{\scala}\left(\frac{\eigvm[\ell']}{2^{j}}\right)}
		\sum_{m=-\ell}^{\ell}\sum_{m'=-\ell'}^{\ell'}
		\left(\dsh\conj{\dsh(\pN)}^{T}+\csh\conj{\csh(\pN)}^{T}\right)\\
		&\hspace{2.2cm}\times \left(\dfco[\ell' m']{\PT{f}}\dsh[\ell' m'](\pN)+\cfco[\ell' m']{\PT{f}}\csh[\ell' m'](\pN)\right)\\
		&\quad= \sum_{\ell=1}^{\infty}\sum_{\ell'=1}^{\infty}\FS{\scala}\left(\frac{\eigvm}{2^{j}}\right)
		\conj{\FS{\scala}\left(\frac{\eigvm[\ell']}{2^{j}}\right)}
		\sum_{m=-\ell}^{\ell}\sum_{m'=-\ell'}^{\ell'}
		\left(\dsh\sum_{k=1}^{N_j}\wN\:\conj{\dsh(\pN)}^{T}\dsh[\ell' m'](\pN)\dfco[\ell' m']{\PT{f}}\right.\\
		&\hspace{1.2cm} +\dsh\sum_{k=1}^{N_j}\wN\:\conj{\dsh(\pN)}^{T}\csh[\ell' m'](\pN)\cfco[\ell' m']{\PT{f}}
		+ \csh\sum_{k=1}^{N_j}\wN\:\conj{\csh(\pN)}^{T}\dsh[\ell' m'](\pN) \dfco[\ell' m']{\PT{f}}\\
		&\hspace{1.2cm}\left. + \csh \sum_{k=1}^{N_j}\wN\:\conj{\csh(\pN)}^{T}\csh[\ell' m'](\pN) \cfco[\ell' m']{\PT{f}}\right).
	\end{align*}
	Then, by \eqref{eq:vsh.InnerD}, \eqref{eq:vsh.InnerD2} and \eqref{eq:cfraprj.f}, for $j=0,1,\dots$,
	\begin{align*}
		&\sum_{k=1}^{N_{j}}\fra\InnerL{\PT{f},\fra}^{T}=\sum_{\ell=1}^{\infty}\left|\FS{\scala}\left(\frac{\eigvm}{2^{j}}\right)\right|^{2}\sum_{m=-\ell}^{\ell}\left(\dfco{\PT{f}}\dsh+\cfco{\PT{f}}\csh\right)=\int_{\sph{2}}\cfra[j,\PT{y}]\InnerL{\PT{f},\cfra[j,\PT{y}]}^{T}\IntDiff{y}.
	\end{align*}
	In a similar way, we obtain that for $n=1,\dots,r$ and $j=0,1,\dots$,
	\begin{align*}
		&\sum_{k=1}^{N_{j}}\frb{n}\InnerL{\PT{f},\frb{n}}^{T}=\sum_{\ell=1}^{\infty}\left|\FS{\scalb^{n}}\left(\frac{\eigvm}{2^{j}}\right)\right|^{2}\sum_{m=-\ell}^{\ell}\left(\dfco{\PT{f}}\dsh+\cfco{\PT{f}}\csh\right)=\int_{\sph{2}}\cfrb[j,\PT{y}]{n}\InnerL{\PT{f},\cfrb[j,\PT{y}]{n}}^{T}\IntDiff{y}.
	\end{align*}
	Thus, (ii) is equivalent to (ii) in Theorem~\ref{thm:equiv.tight.cfrsys}. Then (i), (ii), (iv) and (v) are equivalent. The equivalence between (ii) and (iii) follows from the polarization identity.
\end{proof}

\section{Fast Tensor Needlet Transforms}\label{sec4}
In this section, we discuss the multi-level filter bank transforms
associated with a sequence of tight tensor needlets  $\frsys(\filtbk;\{\QN\}_{j\geq J})$ for $\vLp{2}{2}$. The transforms include the \emph{decomposition} and the \emph{reconstruction}. The decomposition of $\fracoev=(\fracoev[j,k])_{k=1}^{N_j} = (\InnerL{\PT{f},\fra}^{T})_{k=1}^{N_j}$ is split into a coarse scale \emph{approximation coefficient sequence} $\fracoev[j-1]=(\InnerL{\PT{f},\fra[j-1,k]}^{T})_{k=1}^{N_{j-1}}$ and into the coarse scale \emph{detail coefficient sequences} $\frbcoev[j-1]{n}= (\frbcoev[j-1,k]{n} )_{k=1}^{N_j} = (\InnerL{\PT{f},\frb[j-1,k]{n}}^{T})_{k=1}^{N_j}$, $n=1,\ldots,r$. The reconstruction of  $\fracoev$ is an inverse process from the coarse scale approximations and details to fine scales.
We show that the decomposition and reconstruction algorithms for the filter bank transforms can be implemented based on discrete Fourier transforms on $\sph{2}$. In particular, with the employment of fast discrete Fourier transforms (FFTs) for vector spherical harmonics on $\sph{2}$, we are able to develop fast algorithmic realizations for the multi-level filter bank transforms, that is, fast algorithm for tensor needlet transforms.

\subsection{Multi-level filter bank transforms}
\label{sec:decomp.reconstr}

The multi-level framelet filter bank transform algorithms that use convolution, downsampling and upsampling for data sequences on $\sph{2}$, are introduced in this section.

Let $\{\QN\}_{j={\ord[0]}}^{\infty}$ be a sequence of quadrature rules on $\sph{2}$ with $\QN=\{(\wN,\pN) \in\R\times\sph{2} \setsep k=0,\dots,N_{j}\}$ a polynomial-exact quadrature rule of degree $2^j$.
For an integer $N\in\N_0$, let $l(N)$ be the set of complex-valued sequences on $[0,N]$. Let $\Lambda_j:=\dim \Pi_{2^{j-1}} = \#\{\ell\in\N_0: \eigvm\le 2^{j-1}\}$. The following transforms (operators or operations) between sequences in $l(\Lambda_j)$ and sequences in $l(N_j)$ are used to describe and implement the algorithms.
Let $\mask\in l_1(\Z)$ be a mask (filter).
The \emph{discrete convolution}  $\fracoev[] \dconv \mask$ of a sequence $\fracoev[]\in l(\Lambda_j,N_j)$  with a mask $\mask$ is a sequence in  $l(\Lambda_j,N_j)$ defined as
\begin{equation}\label{eq:dis.conv}
    (\fracoev[] \dconv \mask)_k := \sum_{\ell=0}^{\Lambda_j} \: {\FS{\mask}}\left(\frac{\eigvm}{2^{j}}\right)\sqrt{\wN}\:\sum_{m=-\ell}^{\ell}\left(\dfcav[\ell,m]\dsh(\pN)+\tilde{\mathrm{v}}_{\ell,m}\csh(\pN)\right),\quad k=0,\dots,N_{j}.
\end{equation}
The \emph{downsampling operator} $\downsmp:l(\Lambda_j,N_j)\rightarrow l(N_{j-1})$ for a  $(\Lambda_j,N_j)$-sequence $\fracoev[]$ is
\begin{equation}\label{eq:downsampling}
    (\fracoev[]\downsmp)_k := \sum_{\ell=0}^{\Lambda_{j}} \:\sqrt{\wN[j-1,k]}\:\sum_{m=-\ell}^{\ell}\left(\dfcav[\ell,m]\dsh(\pN[j-1,k])+ \tilde{\mathrm{v}}_{\ell,m}\csh(\pN[j-1,k])\right),\quad k=1,\dots,N_{j-1}.
\end{equation}
The \emph{upsampling operator} $\:\upsmp: l(\Lambda_{j-1},N_{j-1})\rightarrow l(\Lambda_j,N_j)$ for a $(\Lambda_{j-1},N_{j-1})$-sequence $\fracoev[]$ is
\begin{equation}\label{eq:upsampling}
(\fracoev[]\upsmp)_k :=
\sum_{\ell=0}^{\Lambda_{j-1}} \sqrt{\wN}\:\sum_{m=-\ell}^{\ell}\left(\dfcav[\ell,m]\dsh(\pN)+\tilde{\mathrm{v}}_{\ell,m}\csh(\pN)\right),\quad k = 1,\ldots, N_{j}.
\end{equation}

For a mask $\mask$, let $\mask^\star$ be the mask satisfying $\FT{\mask^\star}(\xi) = \overline{\FT\mask(\xi)}$, $\xi\in\R$.
The above convolution, downsampling and upsampling give the implementation for decomposition and reconstruction, as follows.
\begin{theorem}
\label{thm:dec:rec}
Let $\frsys(\filtbk;\{\QN\}_{j\geq J})$, $J\geq J_0$, given in \eqref{eq:frsys} be a sequence of tensor needlet systems which elements are given by Definition~\ref{defn:cfr} with the filter bank $\filtbk$ given by \eqref{eq:filtbk}, scaling functions in \eqref{eq:scal.sys} and with a sequence of quadrature rules $\QN=\{(\wN,\pN)\}_{k=1}^{N_j}$ which is exact for degree $2^{j+1}$, $j=0,1,\dots$.
Let $\fracoev=(\fracoev[j,k])_{k=1}^{N_j}$ and $\frbcoev{n}=(\frbcoev[j,k]{n})_{k=1}^{N_{j+1}}$, $n=1,\ldots,r$, be the approximation coefficient sequence and detail coefficient sequences of $f\in\vLp{2}{2}$ at scale $j$ given by
\begin{equation}\label{eq:coe.framelets}
    \fracoev[j,k] := \InnerL{\PT{f},\fra}^{T},\; k = 1,\ldots, N_{j}, \quad\mbox{and}\quad \frbcoev[j,k]{n} := \InnerL{\PT{f},\frb{n}}^{T}, \;  k = 1,\ldots, N_{j+1},\;  n=1,\dots,r,
\end{equation}
respectively. Then,
\begin{enumerate}[{\rm (i)}]
\item the coefficient sequence $\fracoev[j]$ is a  $(\Lambda_j,N_j)$-seqeunce and $\frbcoev[j]{n}, n=1,\ldots, r$, are  $(\Lambda_{j+1},N_{j+1})$-seqeunces for all $j\ge {\ord[0]}$;
\item for any $j\ge {\ord[0]}+1$, the following decomposition relations hold:
\begin{equation}
\label{eq:decomp.j1.j}
  \fracoev[j-1] = (\fracoev\dconv \maska^\star)\downsmp,\quad \frbcoev[j-1]{n} = \fracoev\dconv (\maskb)^\star,\quad  n=1,\ldots,r;
\end{equation}
\item for any $j\ge {\ord[0]}+1$, the following reconstruction relation holds:
\begin{equation}\label{eq:reconstr.j.j1}
  \fracoev =  (\fracoev[j-1]\upsmp) \dconv \maska+\sum_{n=1}^r   \frbcoev[j-1]{n} \dconv \maskb.
\end{equation}
\end{enumerate}
\end{theorem}
\begin{proof}
For $\fracoev$ and $\frbcoev[j-1]{n}$ in \eqref{eq:coe.framelets}, by \eqref{eq:fr.coeff1}, \eqref{eq:fr.coeff2}, \eqref{eq:scaling.mask} and $\supp\FT{\scala}\subseteq[0,1/c]$, we obtain that $\supp\FT{\scalb^n}\subseteq[0,2/c]$,
\[
\fracoev[j,k] =\sum_{\ell=0}^{\Lambda_j}\: \conj{\FT{\scala}\left(\frac{\eigvm}{2^{j}}\right)}\sqrt{\wN}\: \sum_{m=-\ell}^{\ell}\left(\dfco{\PT{f}}\dsh(\pN)+\cfco{\PT{f}}\csh(\pN)\right)
\]
and
\[
\frbcoev[j-1,k]{n} =\sum_{\ell=0}^{\Lambda_j}\: \conj{\FT{\scalb^n}\left(\frac{\eigvm}{2^{j-1}}\right)}\sqrt{\wN}\: \sum_{m=-\ell}^{\ell}\left(\dfco{\PT{f}}\dsh(\pN)+\cfco{\PT{f}}\csh(\pN)\right).
\]
Hence, $\fracoev$ and $\frbcoev[j-1]{n}$, $n=1,\dots,r$, are $(\Lambda_j,N_j)$-sequences with the discrete Fourier coefficients  $\dfcav[j]:=(\dfcav[j,\ell])_{\ell=1}^{\Lambda_{j}}$ and $\dfcbv[j-1]{n}:=(\dfcbv[j-1,\ell]{n})_{\ell=1}^{\Lambda_{j}}$  given by
\begin{equation*}
    \dfcav[j,\ell,m] = \:\dfco{\PT{f}}\: \conj{\FT{\scala}\left(\frac{\eigvm}{2^{j}}\right)},\quad
     \dfcbv[j-1,\ell,m]{n}  = \:\dfco{\PT{f}}\: \conj{\FT{\scalb^n}\left(\frac{\eigvm}{2^{j}}\right)},\quad \ell=1,\ldots,\Lambda_{j},
\end{equation*}
\begin{equation*}
   \tilde{\mathrm{v}}_{j,\ell,m} = \:\cfco{\PT{f}}\: \conj{\FT{\scala}\left(\frac{\eigvm}{2^{j}}\right)},\quad
     \tilde{ \mathrm{w}}_{j-1,\ell,m}^{n}  = \:\cfco{\PT{f}}\: \conj{\FT{\scalb^n}\left(\frac{\eigvm}{2^{j}}\right)},\quad \ell=1,\ldots,\Lambda_{j}.
\end{equation*}
Thus, statement (i) holds.

Note that $\fracoev[j-1]$ is  a $(\Lambda_{j-1},N_{j-1})$-sequence.
For $k=1,\ldots, N_{j-1}$, by \eqref{eq:scaling.mask}, we can obtain
\begin{align*}
    \fracoev[j-1,k]&= \sum_{\ell=1}^{\Lambda_{j-1}}\: \conj{\FT{\scala}\left(\frac{\eigvm}{2^{j-1}}\right)}\sqrt{\wN[j-1,k]}\:\sum_{m=-\ell}^{\ell}\left(\dfco{\PT{f}}\dsh(\pN[j-1.k])+\cfco{\PT{f}}\csh(\pN[j-1.k])\right)\notag\\
    &= \sum_{\ell=1}^{\Lambda_{j-1}}\:\conj{\FT{\scala}\left(\frac{\eigvm}{2^{j}}\right)}\: \conj{\FT{\maska}\left(\frac{\eigvm}{2^{j}}\right)}\sqrt{\wN[j-1,k]}\: \sum_{m=-\ell}^{\ell}\left(\dfco{\PT{f}}\dsh(\pN[j-1.k])+\cfco{\PT{f}}\csh(\pN[j-1.k])\right)\\
    &= \sum_{\ell=1}^{\Lambda_{j}} \:\conj{\FT{\maska}\left(\frac{\eigvm}{2^{j}}\right)}\sqrt{\wN[j-1,k]}\: \sum_{m=-\ell}^{\ell}\left( \dfcav[j,\ell,m]\dsh(\pN[j-1.k])+ \tilde{\mathrm{v}}_{j,\ell,m}\csh(\pN[j-1.k])\right)\\
    &= \bigl((\fracoev \dconv \maska^\star)\downsmp\bigr)_{k}.
\end{align*}
Similarly, for $k=0,\ldots, N_{j-1}$ and $n=1,\dots,r$,
\begin{align*}
\frbcoev[j-1,k]{n}
&=\sum_{\ell=1}^{\Lambda_{j}} \:\conj{\FT{\scalb^n}\left(\frac{\eigvm}{2^{j-1}}\right)}\sqrt{\wN[j-1,k]}\: \sum_{m=-\ell}^{\ell}\left(  \dfcbv[j-1,\ell,m]{n}\dsh(\pN[j-1.k])+ \tilde{ \mathrm{w}}_{j-1,\ell,m}^{n} \csh(\pN[j-1.k])\right)\\
&=(\fracoev \dconv (\maskb)^\star)_k.
\end{align*}
This proves \eqref{eq:decomp.j1.j}, thus, statement (ii) holds.

Using $\fracoev[j-1]=(\fracoev\dconv\maska^{\star})\downsmp$ and $\frbcoev[j-1]{n}=\fracoev \dconv (\maskb)^\star$, we obtain
\[
\breve{\mathrm{v}}: = (\fracoev[j-1]\upsmp)\dconv\maska + \sum_{n=1}^{r} \frbcoev[j-1]{n}\dconv\maskb
=(((\fracoev\dconv\maska^{\star})\downsmp)\upsmp)\dconv\maska
+ \sum_{n=1}^{r} (\fracoev \dconv (\maskb)^\star)\dconv \maskb.
\]
This with \eqref{eq:dis.conv}, \eqref{eq:downsampling}, \eqref{eq:upsampling} and \eqref{eq:filters_union} gives
\[
\begin{aligned}
\breve{\mathrm{v}}_k &=\sum_{\ell=0}^{\Lambda_j}\:\left( \left|\FT\maska\left(\frac{\eigvm}{2^j}\right)\right|^2+\sum_{n=1}^r
 \left|\FT\maskb\left(\frac{\eigvm}{2^j}\right)\right|^2\right) \sqrt{\wN}\sum_{m=-\ell}^{\ell}\left( \dfcav[j,\ell,m]\dsh(\pN)+\tilde{\mathrm{v}}_{j,\ell,m}\csh(\pN)\right)
\\&=\sum_{\ell=0}^{\Lambda_j}\: \sqrt{\wN}\: \sum_{m=-\ell}^{\ell}\left( \dfcav[j,\ell,m]\dsh(\pN)+\tilde{\mathrm{v}}_{j,\ell,m}\csh(\pN)\right)= \fracoev[j,k],\quad
\end{aligned}
\]

thus proving \eqref{eq:reconstr.j.j1}, which completes the proof.
\end{proof}
Equations~\eqref{eq:decomp.j1.j} and \eqref{eq:reconstr.j.j1} show the needlet transform for consecutive levels. Recursively using \eqref{eq:decomp.j1.j} and \eqref{eq:reconstr.j.j1} one obtains the \emph{multi-level needlet transforms}.

\subsection{Fast algorithms for tensor needlet transforms}
In this section, we show the connection of needlet transforms and discrete Fourier transforms for vector spherical harmonics. It provides fast implementation of the needlet transforms with the computational cost nearly proportional to the size of input data or the number of the needlet coefficients at the finest level.

For $j\in\N_0$, we define the \emph{discrete Fourier transform} $\fft[j]: l(\Lambda_j)\rightarrow l(N_{j})$ for two sequences $\mathsf{b}:=\{\mathsf{b}_{\ell,m}: \ell\in \N,\; m=-\ell,\dots,\ell\}$ in $l(\Lambda_j)$ and $\mathsf{c}:=\{\mathsf{c}_{\ell,m}: \ell\in \N,\; m=-\ell,\dots,\ell\}$ in $l(\Lambda_j)$ as
\begin{equation*}\label{def:fft}
\bigl(\fft[j] (\mathsf{b},\mathsf{c})\bigr)_k :=  \sqrt{\wN[j,k]}\sum_{\ell=0}^{\Lambda_j}\sum_{m=-\ell}^{\ell}\left(\mathsf{b}_{\ell,m}\:\dsh(\pN[j,k])+\mathsf{c}_{\ell,m}\:\csh(\pN[j,k])\right),\quad k = 0,\ldots, N_{j}
\end{equation*}
The sequence $\fft[j] (\mathsf{b},\mathsf{c})$ is called a pair of \emph{$(\Lambda_j,N_{j})$-sequences}, and the pair $(\mathsf{b},\mathsf{c})$ is the sequences of \emph{discrete Fourier coefficients} of $\fft[j] (\mathsf{b},\mathsf{c})$. Let $l(\Lambda_j,N_{j})$ be the set of all $(\Lambda_j,N_{j})$-pairs.
The \emph{adjoint discrete Fourier transform} $\fft[j]^*: l(N_{j})\rightarrow l(\Lambda_j)$ for a sequence $\fracoev[]=(\fracoev[k])_{k=0}^{N_j}\in l(N_j)$ is
\begin{equation*}\label{def:adjfft}
(\fft[j]^*\fracoev[])_{\ell,m} :=    \left(\sum_{k=0}^{N_{j}}  \fracoev[k] \sqrt{\wN[j,k]}\:\conj{\dsh(\pN[j,k])},\sum_{k=0}^{N_{j}}  \fracoev[k] \sqrt{\wN[j,k]}\:\conj{\csh(\pN[j,k])}\right), \quad \ell=0,\dots,\Lambda_j,\; m =-\ell,\dots,\ell.
\end{equation*}
Since $\QN$ is a polynomial-exact quadrature rule of degree $2^j$, for every $(\Lambda_j,N_j)$-sequence $\fracoev[]$, there is a unique pair of sequences $\mathsf{b},\mathsf{c}\in l(\Lambda_j)$ such that $\fft(\mathsf{b},\mathsf{c}) = \fracoev[]$.
We can write $(\dfcav[],\tilde{\vv}):=(\mathsf{b},\mathsf{c})=\fft^* \fracoev[]$ for the discrete Fourier coefficient sequence of a $(\Lambda_j,N_j)$-sequence $\fracoev[]$.
Since $\dfco{(\fracoev[] \dconv \mask)} = \dfco{\vv}\:{\FS{\mask}}\left(\frac{\eigvm}{2^{j}}\right)$ and $\cfco{(\vv\dconv \mask)} = \cfco{\vv}\:{\FS{\mask}}\left(\frac{\eigvm}{2^{j}}\right)$ for $\ell\in\Lambda_j$ and $m=-\ell,\dots,\ell$, we obtain that $\widehat{\vv \dconv \mask},\widetilde{\vv\dconv \mask}\in l(\Lambda_j)$ and the discrete convolution in (\ref{eq:dis.conv}) is equivalent to  $\fft(\widehat{\vv \dconv \mask},\widetilde{\vv \dconv \mask})$.

Let $\mathbf{S}(\fft^{*}\vv)$ be the sequence in $l(\Lambda_j)$ given by $\mathbf{S}(\fft^{*}\vv) := \hat{\vv} +\tilde{\vv}$ whose components are
\begin{equation*}
	\sum_{k=0}^{N_{j}}\fracoev[k] \sqrt{\wN[j,k]}\left(\conj{\dsh(\pN[j,k])}+\conj{\csh(\pN[j,k])}\right).
\end{equation*}
We can rewrite \eqref{eq:decomp.j1.j} and \eqref{eq:reconstr.j.j1} using discrete Fourier transforms for vector spherical harmonics, as follows:
\begin{equation}\label{eq:fnt.dft}
\begin{aligned}
  &\fracoev[j-1] = (\fracoev\dconv \maska^\star)\downsmp = \fft[j-1](\widehat{\fracoev\dconv \maska^\star},\widetilde{\fracoev\dconv \maska^\star}),\\
  &\frbcoev[j-1]{n} = \fracoev\dconv (\maskb)^\star = \fft[j-1](\widehat{\fracoev\dconv (\maskb)^\star},\widetilde{\fracoev\dconv (\maskb)^\star}),\quad  n=1,\ldots,r;\\[1mm]
  &
  \fracoev = (\fracoev[j-1]\upsmp) \dconv \maska+\sum_{n=1}^r   \frbcoev[j-1]{n} \dconv \maskb= \left(\mathbf{S}(\fft[j]^*(\fracoev[j-1]))\right) \dconv \maska + \sum_{n=1}^r \left(\mathbf{S}(\fft[j]^*(\frbcoev[j-1]{n}))\right) \dconv \maskb.
\end{aligned}
\end{equation}

The following pseudo-code shows detailed implementation for the decomposition and reconstruction of multi-level FaTeNT.\vspace{3mm}

\IncMargin{1em}
\begin{algorithm}[H]
\SetKwData{step}{Step}
\SetKwInOut{Input}{Input}\SetKwInOut{Output}{Output}
\BlankLine
\Input{$\fracoev[\ord]$ -- a $(\Lambda_J,N_J)$-sequence}

\Output{$\bigl(\{\frbcoev[\ord-1]{n},\frbcoev[\ord-2]{n},\dots,\frbcoev[J_0]{n}\}_{n=1}^{r},\fracoev[J_0]\bigr)$}

$\fracoev[\ord] \longrightarrow \dfcav[\ord]$ \tcp*[f]{adjoint FFT}\\

\For{$j\leftarrow \ord$ \KwTo $J_0+1$}{
    $\dfcav[j-1] \longleftarrow \dfcav[j,\cdot] \: \conj{\FT{\maska}}\left(2^{-j}\eigvm[\cdot]\right)$ \tcp*[f]{downsampling \& convolution}\\[1mm]

\For{$n\leftarrow 1$ \KwTo $r$}{

    $\dfcbv[j-1]{n} \longleftarrow  \dfcav[j,\cdot] \: \conj{\FT{\maskb}}\left(2^{-j}\eigvm[\cdot]\right)$ \tcp*[f]{convolution}\\[1mm]

    $\frbcoev[j-1]{n} \longleftarrow \dfcbv[j-1]{n}$ \tcp*[f]{ FFT}\\
    }
}

$\fracoev[J_0] \longleftarrow \dfcav[J_0]$ \tcp*[f]{ FFT}

\caption{Decomposition of Multi-Level FaTeNT}
\label{algo:decomp.multi.level}
\end{algorithm}
\vspace{3mm}

\begin{algorithm}[H]
\SetKwData{step}{Step}
\SetKwInOut{Input}{Input}\SetKwInOut{Output}{Output}
\BlankLine
\Input{$\bigl(\{\frbcoev[\ord-1]{n},\frbcoev[\ord-2]{n},\dots,\frbcoev[J_0]{n}\}_{n=1}^{r},\fracoev[J_0]\bigr)$}

\Output{$\fracoev[\ord]$ -- a $(\Lambda_{J},N_J)$-sequence}

$\dfcav[J_0] \longleftarrow \fracoev[J_0]$ \tcp*[f]{adjoint   FFT }\\

\For{$j\leftarrow J_0+1$ \KwTo $\ord$}{

\For{$n\leftarrow 1$ \KwTo $r$}{

$\dfcbv[j-1]{n}\longleftarrow \frbcoev[j-1]{n}$ \tcp*[f]{adjoint  FFT}\\

}

$\dfcav \longleftarrow (\dfcav[j-1,\cdot])\:\FT{\maska}\left(2^{-j}\eigvm[\cdot]\right)  + \sum_{n=1}^{r}\dfcbv[j,\cdot]{n} \;\FT{\maskb}\left(2^{-j}\eigvm[\cdot]\right)$\tcp*[f]{upsampling \& convolution}\\

}

$\fracoev[\ord] \longleftarrow \dfcav[\ord]$ \tcp*[f]{FFT}\\
\caption{Reconstruction of Multi-Level FaTeNT}\label{algo:reconstr.multi.level}
\end{algorithm}

\subsection{Fast vector spherical harmonic transforms}
Using the spherical coordinates, the scalar spherical harmonics can be explicitly written as, for $\ell=0,1,\dots$,
\begin{equation*}
\begin{aligned}
  \shY(\PT{x}) &:= \shY(\theta,\varphi) := \sqrt{\frac{2\ell+1}{4\pi}\frac{(\ell-m)!}{(\ell+m)!}}\aLegen{m}(\cos\theta)\: e^{\imu m\varphi}, \quad m=0,1,\ldots, \ell,\\
  \shY(\PT{x}) &:= (-1)^{m}\shY[\ell,-m](\PT{x}), \quad m=-\ell,\dots,-1.
 \end{aligned}
\end{equation*}
In the following, we would suppress the variable $\PT{x}$ in $\shY:=\shY(\PT{x})$ if no confusion arises.

Given the covariant spherical basis vectors \cite{Edmonds2016},
\begin{equation}\label{eq:cov.sph.basis.vec}
\mathbf{e}_{+1} = - \frac{1}{\sqrt{2}} ([1,0,0] + i [0,1,0])^T, \quad
  \mathbf{e}_{0} = [0,0,1]^T, \quad
  \mathbf{e}_{-1} = \frac{1}{\sqrt{2}}([1,0,0] - i [0,1,0])^T,
\end{equation}
the \emph{divergence-free} and \emph{curl-free vector spherical harmonics} can be represented respectively as follows,
\begin{equation}\label{eq:vsh}
\begin{aligned}
 \dsh &= B_{+1,\ell,m} \mathbf{e}_{+1} +
                  B_{0,\ell,m} \mathbf{e}_{0} +
                  B_{-1,\ell,m} \mathbf{e}_{-1},\\
 \csh &=  D_{+1,\ell,m} \mathbf{e}_{+1} +
                   D_{0,\ell,m} \mathbf{e}_{0}  +
                   D_{-1,\ell,m} \mathbf{e}_{-1},
\end{aligned}
\end{equation}
where $\lambda_\ell:=\ell(\ell+1)$ is the eigenvalue of the Laplace-Beltrami operator $\LBo$ for $\shY$ and the associated coefficients are explicitly given by
\begin{equation}\label{eq:Blm}
\begin{aligned}
  B_{+1,\ell,m} &= c_{\ell} C^{\ell,m}_{\ell-1,m-1,1,1} \shY[\ell-1,m-1] + d_\ell C^{\ell,m}_{\ell+1,m-1,1,1} \shY[\ell+1,m-1]\\
  B_{0,\ell,m} &= c_\ell C^{\ell,m}_{\ell-1,m,1,0} \shY[\ell-1,m] + d_{\ell} C^{\ell,m}_{\ell+1,m,1,0} \shY[\ell+1,m]\\
  B_{-1,\ell,m} &= c_\ell C^{\ell,m}_{\ell-1,m+1,1,-1}\shY[\ell-1,m+1] + d_\ell C^{\ell,m}_{\ell+1,m+1,1,-1}\shY[\ell+1,m+1]\\
  D_{+1,\ell,m} &= i C^{\ell,m}_{\ell,m-1,1,1} \shY[\ell,m-1], \\
  D_{0,\ell,m} &= i C^{\ell,m}_{\ell,m,1,0}\shY[\ell,m],\\
  D_{-1,\ell,m}&= i C^{\ell,m}_{\ell,m+1,1,-1}\shY[\ell,m+1],
\end{aligned}
\end{equation}
where the Clebsch-Gordan (CG) coefficients
\begin{equation*}
 C^{\ell,m}_{j_1,m_1,j_2,m_2} :=
(-1)^{(m+j_1-j_2)}\sqrt{2\ell+1}\left(\begin{array}{lll}j_1 & j_2 & \ell \\ m_1 &  m_2 & -m
\end{array} \right),
\end{equation*}
and
\begin{equation}\label{eq:cl,dl,betal}
	c_{\ell}:=\sqrt{\frac{\ell+1}{2\ell+1}},\quad d_{\ell}:=\sqrt{\frac{\ell}{2\ell+1}}.
\end{equation}
The Clebsch-Gordan coefficients are most familiar in quantum mechanics as the coefficients of a unitary
transformation that connects the tensor product basis from irreducible representations of SO(3), to its irreducible components in a direct sum, labelled by total angular momentum invariants \cite{Edmonds2016,biedenharn1984angular}. However, the more abstract algebraic theory of algebraic invariants by Clebsch and Gordon, was developed earlier in the 19th Century \cite{gordon1885vorlesungen}.

By \eqref{eq:cov.sph.basis.vec} and \eqref{eq:vsh}, for $\ell\ge1, m=-\ell,\dots,\ell$, the divergence-free and curl-free vector spherical harmonics of degree $(\ell,m)$ are
\begin{align}\label{eq:vsh2}
	\dsh = \begin{pmatrix}
		-\frac{1}{\sqrt{2}} \left(B_{+1,\ell,m} - B_{-1,\ell,m}\right)\\
		-\frac{1}{\sqrt{2}}\imu \left(B_{+1,\ell,m} + B_{-1,\ell,m}\right)\\
		B_{0,\ell,m}
	\end{pmatrix}
	\qquad
	\csh = \begin{pmatrix}
		-\frac{1}{\sqrt{2}}\left(D_{+1,\ell,m}-D_{-1,\ell,m}\right)\\
		-\frac{1}{\sqrt{2}}\imu \left(D_{+1,\ell,m}+D_{-1,\ell,m}\right)\\
		D_{0,\ell,m}
	\end{pmatrix}.
\end{align}
The set of vector spherical harmonics $\{\dsh,\csh: \ell=1,2,\dots, m=-\ell,\dots,\ell\}$ which is used in quantum mechanics \citep{Edmonds2016} forms an orthonormal basis for $\vLp{2}{2}$.
Based on properties detailed in \citet{NIST:DLMF}, we can obtain the explicit expression of those (CG) coefficients, respectively, see details in \cite{FaVest2019}.
By \eqref{eq:vsh2}, the fast vector spherical harmonics transforms can be implemented by using FFTs for scalar spherical harmonics multiple times \cite{GaLeSl2011,KeKuPo2007,RoTy2006, Tygert2008,Tygert2010}. Due to the space limitation, we refer the interested readers to \cite{FaVest2019} for details.

\subsection{Computational complexity analysis}
In the relation \eqref{eq:fnt.dft} the discrete convolution has the computational cost $\bigo{}{N_j}$ for level~$j$. FFTs for vector spherical harmonics have computational complexity $\bigo{}{N\sqrt{\log N}}$ for input data of size $N$, $N\ge1$. For decomposition and reconstruction in \eqref{eq:fnt.dft}, the input data has length up to $N_J$. Thus, the computational cost of the needlet transforms, which would be dominated by those of FFTs, is $\bigo{}{N_J}$. Hence, we call the needlet transforms in \eqref{eq:fnt.dft} and their multi-level versions \emph{fast tensor needlet transforms} or FaTeNT for tangent fields on $\sph{2}$.

\section{Numerical Study}\label{sec5}
\label{sec:fmtS2expr}
In this section, we start with an explicit construction of tensor needlets on $\sph{2}$, then introduce two types of quadrature rule nodes and two kinds of tangent fields used in our experiments. Later, we describe the experimental setup for using FaTeNT in our simulations. After that, we present the simulation results and discuss the effectiveness and efficiency of our proposed method.
\subsection{Tensor needlets construction}
\label{sec:fmtS2}
Here we demonstrate the procedures of tensor needlets construction, as similar to the steps performed in \cite{WaZh2018}. Based on our theoretical results, the construction needs several components: an explicit formulation of filter banks and the associated generators, vector spherical harmonics and a set of quadrature rule nodes for the discretization of continuous integral.

First, for the filter bank, let us consider a simple one with two high-pass filters, that is, $\filtbk=\{\maska;\maskb[1],\maskb[2]\}$, with the Fourier series denoted by

\begin{subequations}\label{eqs:mask.numer.s3}
\begin{align}
  \FT{\maska}(\xi)&: =
  \left\{\begin{array}{ll}
    1, & |\xi|<\frac{1}{8},\\[1mm]
    \cos\bigl(\frac{\pi}{2}\hspace{0.3mm} \nu(8|\xi|-1)\bigr), & \frac{1}{8} \le |\xi| \le \frac{1}{4},\\[1mm]
    0, & \frac14<|\xi|\le\frac12,
    \end{array}\right.\nonumber \\[1mm]
  \FT{\maskb[1]}(\xi)&:  =\left\{\begin{array}{ll}
    0, & |\xi|<\frac{1}{8},\\[1mm]
    \sin\bigl(\frac{\pi}{2}\hspace{0.3mm} \nu(8|\xi|-1)\bigr), & \frac{1}{8} \le |\xi| \le \frac{1}{4},\\[1mm]
    \cos\bigl(\frac{\pi}{2}\hspace{0.3mm} \nu(4|\xi|-1)\bigr), & \frac14<|\xi|\le\frac12.
    \end{array}\right.\nonumber \\[1mm]
  \FT{\maskb[2]}(\xi)&:
  =\left\{\begin{array}{ll}
    0, & |\xi|<\frac{1}{4},\nonumber\\[1mm]
    \sin\bigl(\frac{\pi}{2}\hspace{0.3mm} \nu(4|\xi|-1)\bigr), & \frac{1}{4} \le |\xi| \le \frac{1}{2},
    \end{array}\right.
\end{align}
\end{subequations}
where
\begin{equation*}
  \nu(t) := \scalg[3](t)^{2} = t^{4}(35 - 84t + 70t^{2} - 20 t^{3}), \quad t\in\Rone,
\end{equation*}
as in \cite[Chapter~4]{Daubechies1992}. It holds that for all $\xi\in[0,1/2]$, $|\FT{\maska}(\xi)|^2+|\FT{\maskb[1]}(\xi)|^2+|\FT{\maskb[2]}(\xi)|^2 = 1$, which implies \eqref{eq:filters_union}. Therefore, the associated generators $\Psi=\{\scala; \scalb^1,\scalb^2\}$ that satisfy \eqref{eq:scaling.mask} and \eqref{eq:scal.j.j+1} can be explicitly defined by
\begin{subequations}\label{eqs:scal.numer.s3}
\begin{align}
  \FT{\scala}(\xi)&=
  \left\{\begin{array}{ll}
    1, & |\xi|<\frac{1}{4},\nonumber\\[1mm]
    \cos\bigl(\frac{\pi}{2}\hspace{0.3mm} \nu(4|\xi|-1)\bigr), & \frac{1}{4} \le |\xi| \le \frac{1}{2},\\[1mm]
    0, & \hbox{else},
    \end{array}\right. \nonumber\\[1mm]
  \FT{\scalb^1}(\xi)&= \left\{\begin{array}{ll}
    \sin\left(\frac{\pi}{2}\hspace{0.3mm} \nu(4|\xi|-1)\right), & \frac{1}{4}\le|\xi|<\frac{1}{2},\\[1mm]
    \cos^2\left(\frac{\pi}{2}\hspace{0.3mm} \nu(2|\xi|-1)\right), & \frac{1}{2} \le |\xi| \le 1,\\[1mm]
    0, & \hbox{else},
    \end{array}\right.\nonumber\\[1mm]
  \FT{\scalb^2}(\xi)&= \left\{\begin{array}{ll}
   0, &|\xi|<\frac{1}{2},\\[1mm]
    \cos\left(\frac{\pi}{2}\hspace{0.3mm} \nu(2|\xi|-1)\right) \sin\left(\frac{\pi}{2}\hspace{0.3mm} \nu(2|\xi|-1)\right), & \frac{1}{2} \le |\xi| \le 1,\nonumber\\[1mm]
    0, & \hbox{else}.
    \end{array}\right.
\end{align}
\end{subequations}
Then, $\FS{\maska},\FS{\maskb[1]},\FS{\maskb[2]},\FT{\scala},\FT{\scalb^{1}},\FT{\scalb^{2}}$ are all in $\CkR[4-\epsilon]$ for some small positive $\epsilon$ \cite[p.~119]{Daubechies1992}, $\supp\FT{\scala}\subseteq [0,1/2]$ and $\supp\FT{\scalb^{n}}\subseteq [1/4,1]$, $n=1,2$. Also, the refinable function $\FT{\scala}$ satisfies \eqref{eq:scala.lim.semidis}. We omit the visualization of the filters $\FS{\maska}$, $\FS{\maskb[1]}$, $\FS{\maskb[2]}$ and the functions $\FT{\scala}$, $\FT{\scalb^1}$, $\FT{\scalb^2}$. Readers can refer to \cite{WaZh2018} for the plotting.

The continuous tensor needlets $\cfra(\PT{x}), \cfrb{1}(\PT{x})$ and $ \cfrb{2}(\PT{x})$ on $\sph{2}$ are expressed by
\begin{align*}
	\cfra(\PT{x}) &:= \sum_{\ell=1}^{\infty} \FS{\scala}\left(\frac{\ell}{2^{j}}\right)\sum_{m=-\ell}^{\ell}\Bigl(\dsh(\PT{x})\otimes\dsh(\PT{y}) + \csh(\PT{x})\otimes\csh(\PT{y})\Bigr),\\
	\cfrb{n}(\PT{x}) &:= \sum_{\ell=1}^{\infty} \FS{\scalb^{n}}\left(\frac{\ell}{2^{j}}\right)\sum_{m=-\ell}^{\ell}\Bigl(\dsh(\PT{x})\otimes\dsh(\PT{y}) + \csh(\PT{x})\otimes\csh(\PT{y})\Bigr), \quad n=1,\dots,2,
\end{align*}
where the vector spherical harmonics $\dsh$, $\csh$ are defined in (\ref{eq:vsh}).

Based on Theorem~\ref{thm:equiv.tight.cfrsys} and the construction of $\Psi$ and $\filtbk$ in \eqref{eqs:scal.numer.s3} and \eqref{eqs:mask.numer.s3}, one can easily verify that the continuous tensor needlet system $\cfrsys(\Psi)=\{\cfra[J,\PT{y}];\cfrb{1},\cfrb{2}: j\geq J, \PT{y}\in\sph{2}\}$ on $\sph{2}$ is a tight tensor needlet system for $\Lpm[\sph{2}]{2}$ for any $J\in\Z$.

By using a polynomial-exact quadrature rule $\QN$ on $\sph{2}$, we can obtain the discrete framelets $\fra(\PT{x}), \frb{1}(\PT{x})$ and $\frb{2}(\PT{x})$, that is
\begin{align*}
	\fra(\PT{x}) &:= \sqrt{\wN}\sum_{\ell=1}^{\infty} \FS{\scala}\left(\frac{\eigvm}{2^{j}}\right)\sum_{m=-\ell}^{\ell}\left(\dsh(\PT{x})\otimes\dsh(\pN) + \csh(\PT{x})\otimes\csh(\pN)\right),\\
   \frb[j,k']{n}(\PT{x}) &:=\sqrt{\wN[j+1,k']}\sum_{\ell=1}^{\infty} \FS{\scalb^{n}}\left(\frac{\eigvm}{2^{j}}\right)\sum_{m=-\ell}^{\ell}\left(\dsh(\PT{x})\otimes\dsh(\pN[j+1,k']) + \csh(\PT{x})\otimes\csh(\pN[j+1,k'])\right).
\end{align*}

Then, we have $\fra\in\Pi_{2^{j-1}}$ and $\frb{1}, \frb{2}\in\Pi_{2^{j}}$ because that the supports of $\FT\scala$, $\FT{\scalb^1}$ and $\FT{\scalb^2}$ are subsets of $[0,1/2]$, $[0,1]$ and $[0,1]$.
Given $\QN$ a polynomial-exact quadrature rule of degree $2^j$ for all $j\in\Z$, the tensor needlet system $\frsys(\filtbk;\{\QN\}_{j\geq J}):=\{\fra[J,k];\frb[j,k']{1},\frb[j,k']{2}: j\geq J,\: k=1,\dots,N_{J},\: k'=1,\dots,N_{j+1}\}$ is a semi-discrete tight needlet system for $\Lpm[\sph{2}]{2}$ for all $J\in\Z$.

\subsection{Data points}\label{sec:numer}
In our simulations, we consider two types of polynomial-exact quadrature rule point sets, described as follows
\begin{enumerate}
  \item \emph{Gauss-Legendre tensor product rule} (GL) \cite{HeWo2012}. The Gauss-Legendre tensor product rule is a (polynomial-exact but not equal area) quadrature rule $\QN:=\{(w_i,\bx_i)\}_{i=1}^{N}, i=0,\ldots,N\}$ on the sphere generated by the tensor product of the Gauss-Legendre nodes on the interval $[-1,1]$ and equi-spaced nodes on the longitude with non-equal weights. Figure~\ref{fig:QN}(a) shows $N=512$ GL points.
  \item \emph{Symmetric spherical designs} (SD) \cite{Womersley_ssd_URL}. The symmetric spherical design is a (polynomial-exact) quadrature rule $\QN:=\{(w_i,\bx_i)\}_{i=1}^{N}, i=0,\ldots,N\}$ on the sphere $\sph{2}$ with equal weights $w_i=1/N$. The points are ``equally'' distributed on the sphere.  Figure~\ref{fig:QN}(b) shows $N=498$ SD points.
\end{enumerate}

\begin{figure}[th]
\begin{minipage}{\textwidth}
  \centering
  \begin{minipage}{0.35\textwidth}
  \centering
  \includegraphics[trim = 0mm 0mm 0mm 0mm, width=0.81\textwidth]{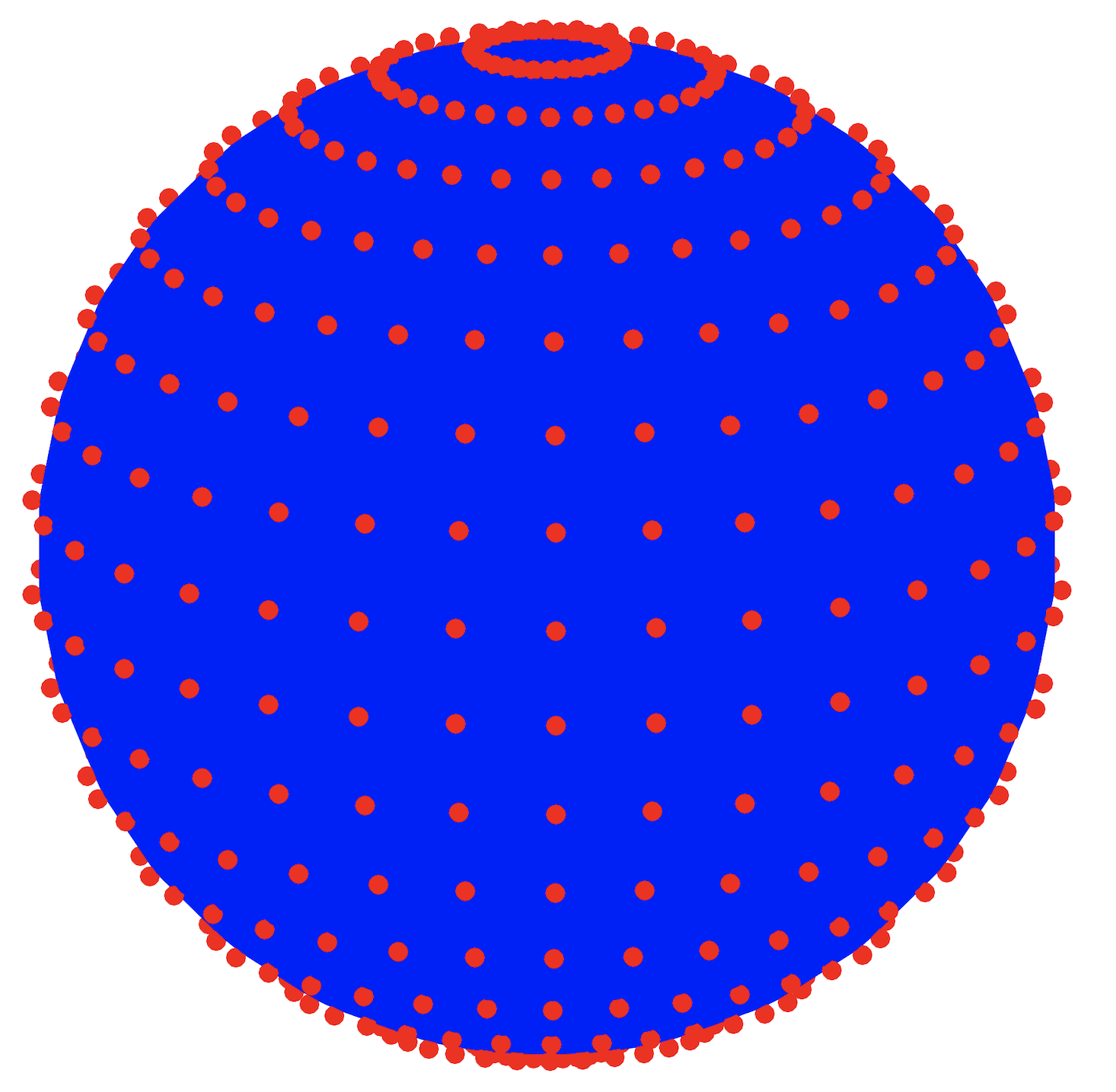}\\[1mm]
  \subcaption{GL, $N=512$}\label{fig:GL}
  \end{minipage}\hspace{1mm}
  \begin{minipage}{0.35\textwidth}
  \centering
  \includegraphics[trim = 0mm 0mm 0mm 0mm, width=0.81\textwidth]{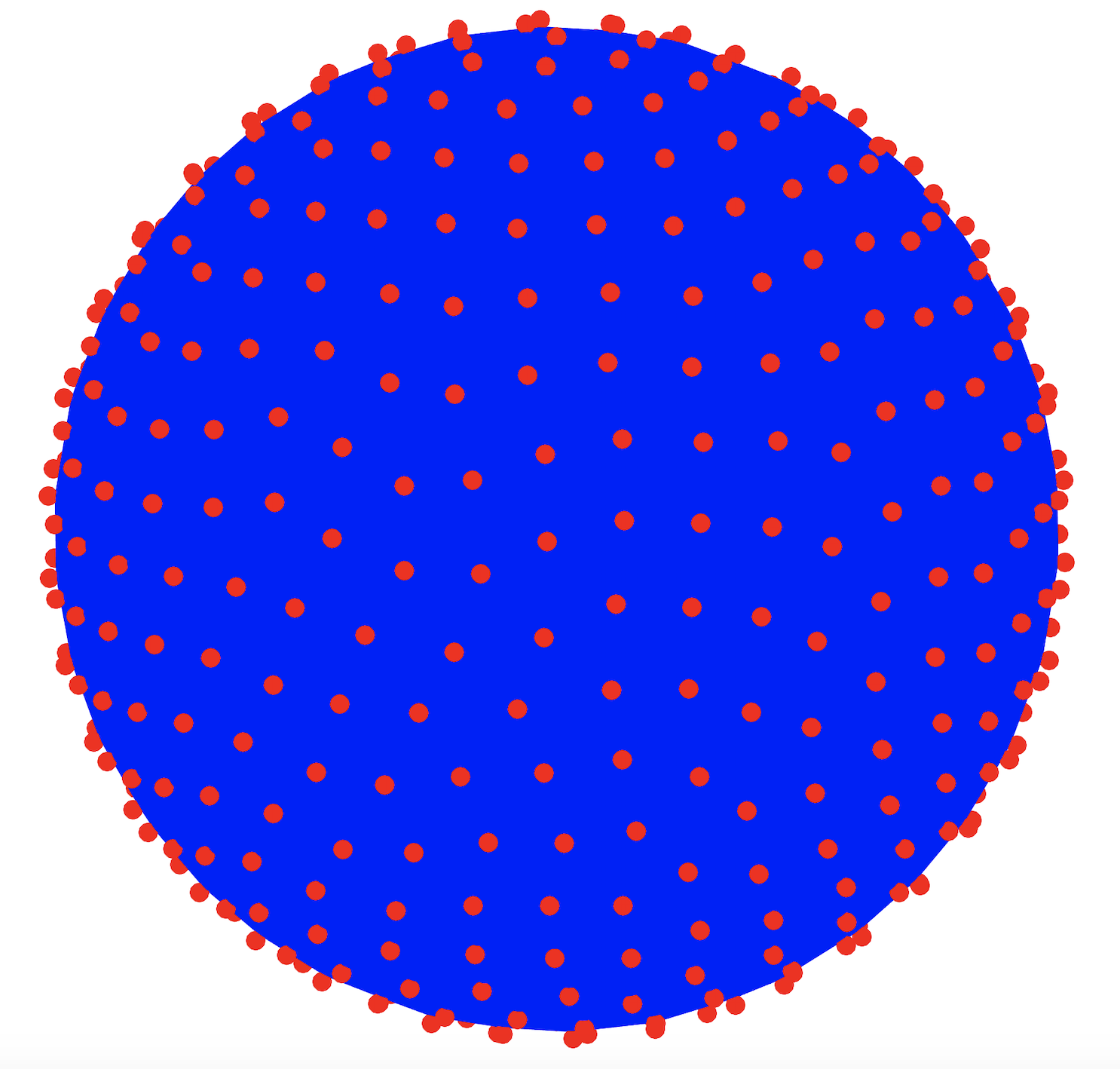}\\[1mm]
  \subcaption{SD, $N=498$}\label{fig:SD}
  \end{minipage}
\end{minipage}
\caption{Point sets on the sphere for: (a) Nodes of Gauss-Legendre tensor rule (GL), (b) Nodes of symmetric spherical designs (SD)}
\label{fig:QN}
\end{figure}

\subsection{Synthetic tangent fields on $\sph{2}$}
To verify our theoretical results and the proposed algorithms, we use two types of simulated tangent fields provided in \cite{FrSc2009}. All these fields are generated using ``stream functions'' and ``velocity potentials'' so that we can easily separate the divergence-free and curl-free parts of the field. Denoted by $s$ and $v$ the stream function and velocity potential, respectively, then each of the tangent fields can be represented by
\begin{equation*}
  T=\underbrace{\mathbf{L}s}_{f^{\rm div}}+\underbrace{\nabla_{*}v}_{f^{\rm curl}}.
\end{equation*}
Here the $\mathbf{L}$ and $\nabla_{*}$ are the surface curl and surface gradient, see Section~\ref{}, and the two terms $\mathbf{L}s$ and $\nabla_{*}v$ are divergent-free and curl-free respectively.
We show the detailed formulism of these two tangent fields, as follows.
\paragraph{Tangent Field A} The stream function and velocity potential for this field are linear combinations of spherical harmonics and are meant to generate realistic synoptic scale meteorological wind fields \cite{FrSc2009}.
The stream function is defined by
\begin{equation}\label{RosbyHaurwitz}
s_1(\PT{x})=-\frac{1}{\sqrt{3}}Y_{1,0}(\PT{x})+\frac{8\sqrt{2}}{3\sqrt{385}}Y_{5,4}(\PT{x}),
\end{equation}
which is known as a Rosby--Haurwitz wave and is an analytic solution to the nonlinear barotropic vorticity equation on the sphere \cite[pp.~453--454]{Holton1973}. In \cite{Williamson1992}, $s_1$ was used as the initial condition for one of the de facto test cases for the shallow water wave equations on the sphere.
The velocity potential is given by
\begin{equation*}
  v_1(\PT{x})=\frac{1}{25}(Y_{4,0}(\PT{x})+Y_{6,-3}(\PT{x})).
\end{equation*}
Note that we can choose different orders of the the spherical harmonics ($l, m$) and coefficients in the above formula of scalar potentials. Here, we inherit the setting of \cite{FrSc2009}.
\paragraph{Tangent Field B} This field still uses the Rosby--Haurwitz wave (\ref{RosbyHaurwitz}) as the stream function, while a linear combination of compactly supported functions for the velocity potential, i.e.
\begin{equation*}
  v_2(\PT{x})=\frac{1}{8}f(\PT{x};5,\pi/6,0)-\frac{1}{7}f(\PT{x};3,\pi/5,\pi/7)+\frac{1}{9}f(\PT{x};5,-\pi/6,\pi/2)-\frac{1}{8}f(\PT{x};3,-\pi/5,\pi/3),
\end{equation*}
with
\begin{equation*}
  f(\PT{x};\sigma,\theta_c,\lambda_c)=\frac{\sigma^3}{12}\sum_{j=0}^{4}(-1)^j\left(
\begin{array}{c}
4\\
j
\end{array} \right)\left|r-\frac{(j-2)}{\sigma}\right|^3.
\end{equation*}
\paragraph{Tangent Field C} Let $\PT{x}_c\in S^2$ have spherical coordinates $(\theta_c,\lambda_c)$, and let $t=\PT{x}\cdot \PT{x}_c$ and $a=1-t$. Define
\begin{equation*}
  g(\PT{x};\theta_c,\lambda_c)=-\frac{1}{2}((3t+3\sqrt{2}a^{3/2}-4)+(3t^2-4t+1)\log(a)+(3t-1)a\log(\sqrt{2a}+a)).
\end{equation*}
The stream function for this tangent field is given by
\begin{equation*}
  s_3(x)=\int_{-\pi/2}^{\theta}\sin^{14}(2\xi)d\xi-3g(x;\pi/4,-\pi/12),
\end{equation*}
where $\theta$ denotes the latitudinal coordinate for $\PT{x}$.
Using $g$, the velocity potential is defined as
\begin{equation*}
  v_3(\PT{x})=\frac{5}{2}g(\PT{x};\pi/4,0)-\frac{7}{4}g(\PT{x};\pi/6,\pi,9)-\frac{3}{2}g(\PT{x};5\pi/16,\pi/10).
\end{equation*}
The left columns of Figure \ref{fig:reconstruction_gl} and \ref{fig:reconstruction_sd} present the tangent fields evaluated at $N=1922$ GL points and $N=1849$ SD points respectively, where the blue arrows and the length of the arrows indicate the direction and scalar value of the tangent field.
\subsection{Experimental setup}
Based on the developed theoretical results and methodology, we detail the main steps performed in the experimental study. First, to evaluate the simulated tangent fields by means of Algorithm 1 and 2, we construct a semi-discrete tensor needlet system given by \eqref{eq:frsys}, where $\Psi=\{\scala; \scalb^1,\scalb^2\}$ are the generators associated with the filter bank $\filtbk=\{\maska;\maskb[1],\maskb[2]\}$ introduced in Subsection~\ref{sec:fmtS2} with $\QN$ a sequence of quadrature rule point sets (either GL or SD). To meet the condition of the decomposition and reconstruction algorithms, the data sequence $v$ sampled from the tangent fields at $\QN[N_J]$ (at the finest scale $J$) must be a $(\Lambda_j,N_j)$-sequence. To achieve this, we project $\fracoev[]$ onto $\Pi_{2^{J}}$ by performing $\fracoev[\ord] = (\fft[\ord]^*\fft[\ord])^{-1}\fft[\ord]^*\fracoev[]$, where $\fft[\ord]$ and $\fft[\ord]^*$ can be implemented fast by forward and adjoint FaVeST algorithms \cite{FaVest2019}. In this manner, we obtain an approximation coefficient sequence $\fracoev[\ord]$ (which is a $(\Lambda_j,N_j)$-sequence) at the finest scale  and the projection error sequence $\projerr=\fracoev[]-\fracoev[J]$.

Recall that GL and SD are a polynomial-exact quadrature rule of order $2^j$ for all $j=J_0,\ldots, J$. Then, by Theorem~\ref{thm:dec:rec}, we can implement decomposition of FaTeNT (Algorithm \ref{algo:decomp.multi.level}) on the $(\Lambda_J,N_J)$-sequence $\fracoev[J]$ and obtain the coefficient sequences $\frbcoev[\ord-1]{1}$, $\frbcoev[\ord-1]{2}$, $\ldots$, $\frbcoev[{\ord[0]}]{1}$, $\frbcoev[{\ord[0]}]{2}$, $\fracoev[{\ord[0]}]$. From the decomposed coefficient sequences $(\frbcoev[\ord-1]{1},\dots,$ $\frbcoev[\ord-1]{r}, \ldots, \frbcoev[{\ord[0]}]{1},\dots,\frbcoev[{\ord[0]}]{r}, \fracoev[{\ord[0]}])$, we are able to exactly reconstruct $\fracoev[\ord]$ by using the reconstruction of FaTeNT Algorithm \eqref{algo:reconstr.multi.level}. Once $\fracoev[\ord]$ is obtained, the sequence $\fracoev[]=\fracoev[\ord]+\projerr$ can be computed with the pre-computed projection error $\projerr$.

\subsection{Results and discussion}
The middle columns of Figure \ref{fig:reconstruction_gl} and \ref{fig:reconstruction_sd} show the reconstructed fields $T^{\rm rec}$ for the original tangent fields $T$ by one level FaTeNT with $J_0=4,J=5$.
We have used $N=2178$ GL points and $N=2148$ SD points, respectively. The corresponding point-wise errors $T-T^{\rm rec}$ on the evaluation points are visualized in the right columns. Also, the maximum norm of the target fields, the reconstructed tangent fields, and the error fields are displayed for each case. It is clear that one level FaTeNT works effectively on these tangent fields. Approximation errors for Tangent Field A (on either GL and SD points) are smaller than that of Tangent Field B due to its better smoothness. As expected, the reconstruction error becomes smaller if more quadrature rule points (or equivalently, higher level $J$) are used. For a better resolution, we only demonstrate the visualization of the case of $J=5$.

To verify the computational consumption of Algorithms 1 and 2, we test for the multi-level FaTeNT with GL points for $J=5,6,\ldots,11$ and $J_0=1$ on Tangent Field A. In the simulations $N:=N_J\approx2^{2J+1}$ GL nodes are used in the implementation of the multi-level FaTeNT decomposition, and $M:=M_J\approx2^{2J}$ coefficients are used in the multi-level FaTeNT reconstruction. For $J=5,6,\ldots,11$, the CPU time for the FaTeNT decomposition and reconstruction are reported in Table~\ref{tab:timecost}. It also displays the ratio of the times for consecutive levels $J$ and $J-1$ to reflect the changing rate of the computational cost. Correspondingly, Figure~\ref{fig:timecost}(a) and (b) show the near linear computational complexity of the FaTeNT decomposition and reconstruction. The cost of FaTeNT decomposition and reconstruction are approximately proportional to $\mathcal{N}^{1.1}$ and $\mathcal{M}^{1.1}$ respectively. These numerical results illustrate that the FaTeNT algorithms incur near linear computational cost.
\begin{figure}[htbp!]
\centering
\begin{minipage}{\textwidth}
  \subcaptionbox{Reconstruction of Tangent Field A with GL}
  {\includegraphics[width=0.31\linewidth]{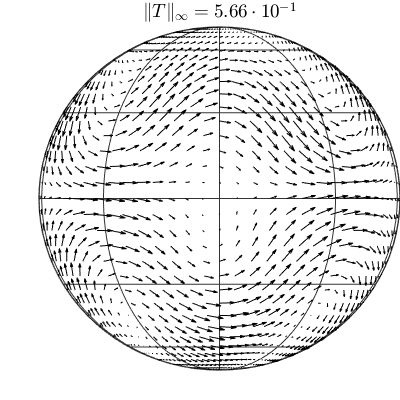}\quad
    \includegraphics[width=0.31\linewidth]{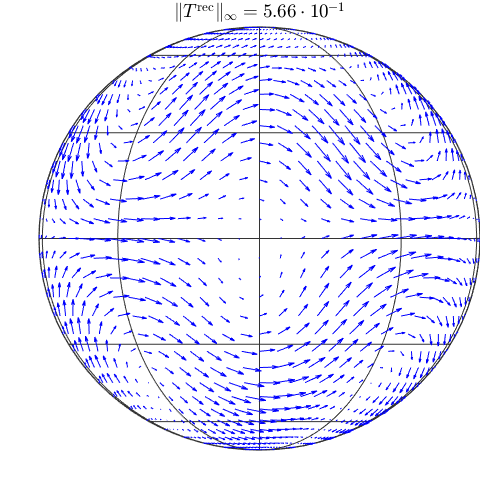}\quad
    \includegraphics[width=0.31\linewidth]{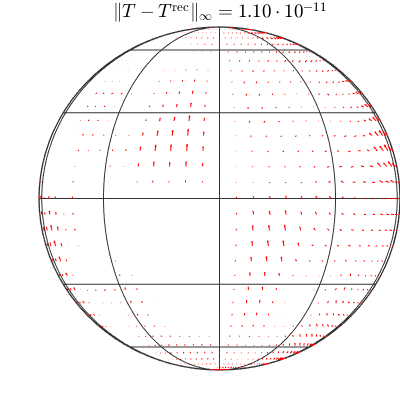}}\\[4mm]
    \subcaptionbox{Reconstruction of Tangent Field B with GL}
  {\includegraphics[width=0.31\linewidth]{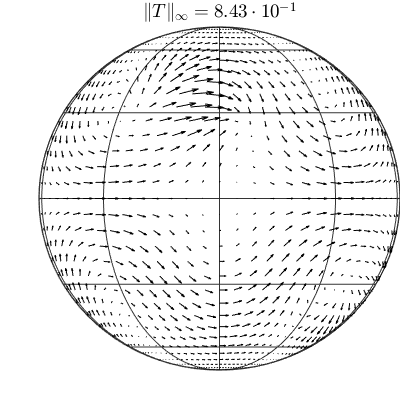}\quad
    \includegraphics[width=0.31\linewidth]{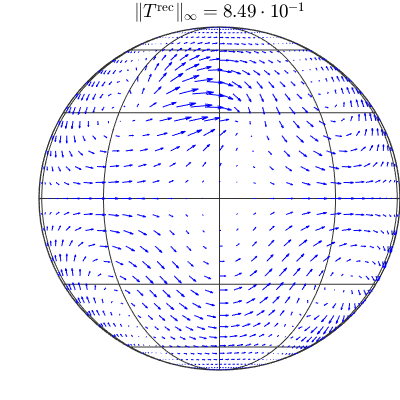}\quad
    \includegraphics[width=0.31\linewidth]{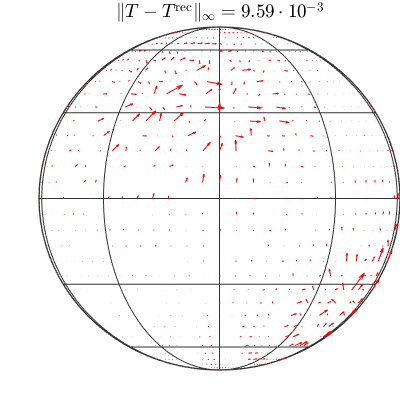}}\\[4mm]
    \subcaptionbox{Reconstruction of Tangent Field C with GL}
  {\includegraphics[width=0.31\linewidth]{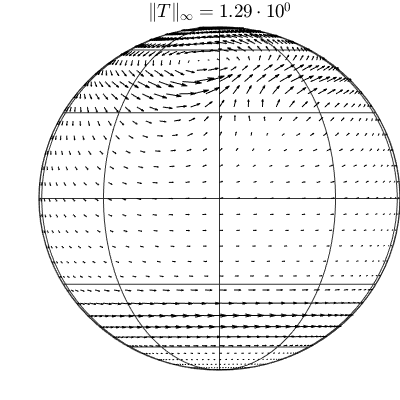}\quad
    \includegraphics[width=0.31\linewidth]{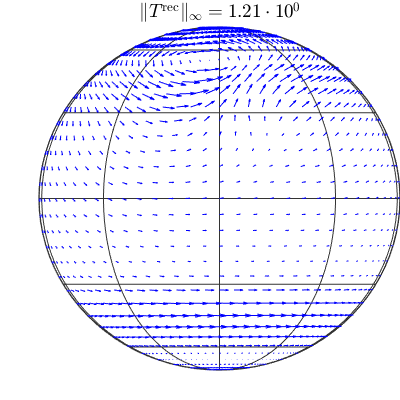}\quad
    \includegraphics[width=0.31\linewidth]{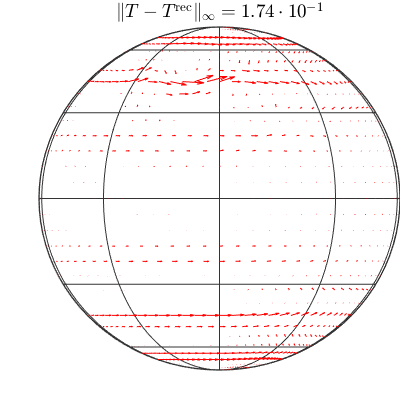}}
\end{minipage}
    \caption{Visualization of the vector field reconstruction by one level FaTeNT with GL. The first column shows the target field $T$. The second and third columns are the reconstructed field $T^{\rm rec}$ and point-wise error $T-T^{\rm rec}$. All plots are orthographic projections of the fields evaluated at level $J=5$ with $N = 2178$ GL nodes. The normalized max norms for $T$,  $T^{\rm rec}$ and $T-T^{\rm rec}$ are displayed for each case.}\label{fig:reconstruction_gl}
\end{figure}

\begin{figure}[htbp!]
\centering
\begin{minipage}{\textwidth}
 \subcaptionbox{Reconstruction of Vector Field 1 with SD points}
{\includegraphics[width=0.31\linewidth]{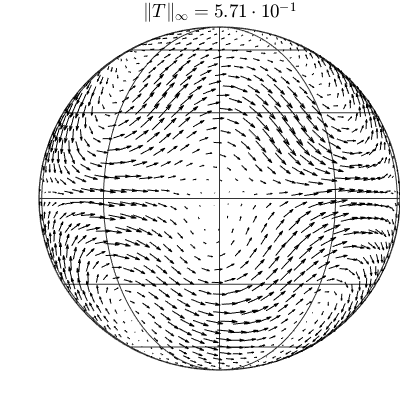}\quad
    \includegraphics[width=0.31\linewidth]{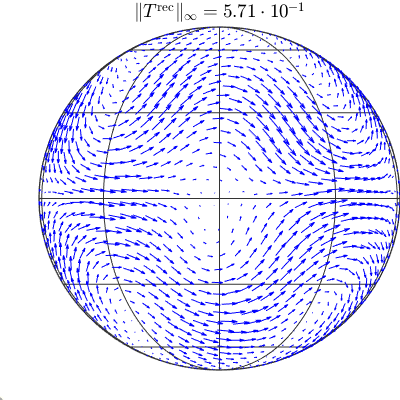}\quad
    \includegraphics[width=0.31\linewidth]{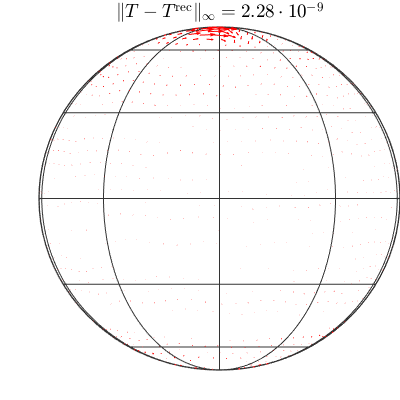}}\\[4mm]
    \subcaptionbox{Reconstruction of Vector Field 2 with SD points}
  {\includegraphics[width=0.31\linewidth]{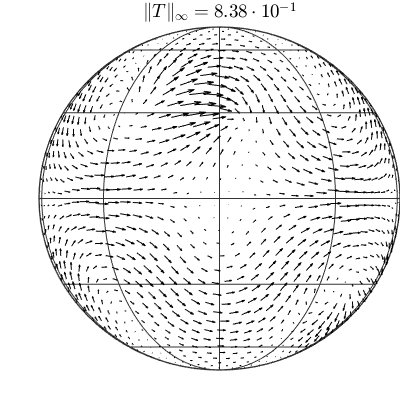}\quad
    \includegraphics[width=0.31\linewidth]{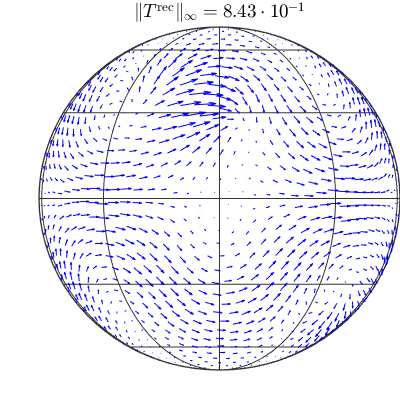}\quad
    \includegraphics[width=0.31\linewidth]{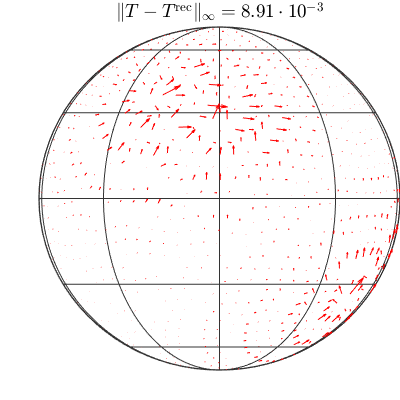}}\\[4mm]
    \subcaptionbox{Reconstruction of Vector Field 3 with SD points}
  {\includegraphics[width=0.31\linewidth]{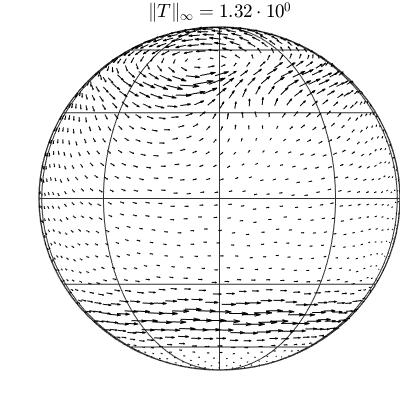}\quad
    \includegraphics[width=0.31\linewidth]{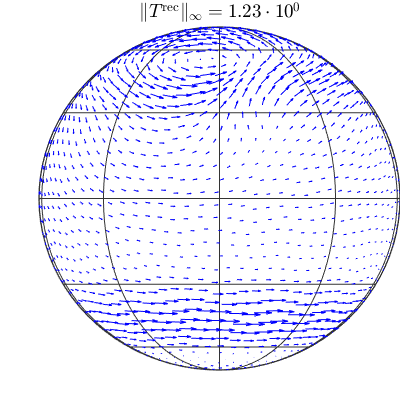}\quad
    \includegraphics[width=0.31\linewidth]{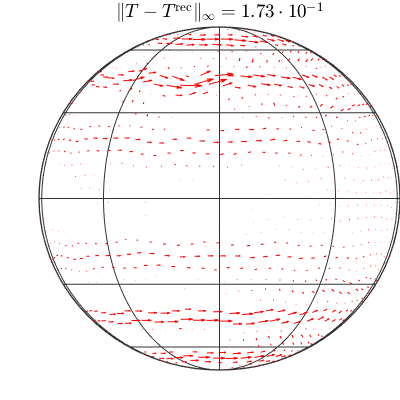}}
\end{minipage}
    \caption{Visualization of the vector field reconstruction by one level FaTeNT with SD points. The first column shows the target tangent field $T$. The second and third columns are the reconstructed field $T^{\rm rec}$ and point-wise error $T-T^{\rm rec}$. All plots are orthographic projections of the fields evaluated at $N = 2148$ SD nodes in level $J=5$. The normalized max norms for $T$, $T^{\rm rec}$ and $T-T^{\rm rec}$ are displayed for each case.}\label{fig:reconstruction_sd}
\end{figure}

\begin{figure}[htbp!]
\begin{minipage}{\textwidth}
	\centering
\begin{minipage}{\textwidth}
  \centering
  \begin{minipage}{0.46\textwidth}
  \centering
  \includegraphics[trim = 0mm 0mm 0mm 0mm, width=\textwidth]{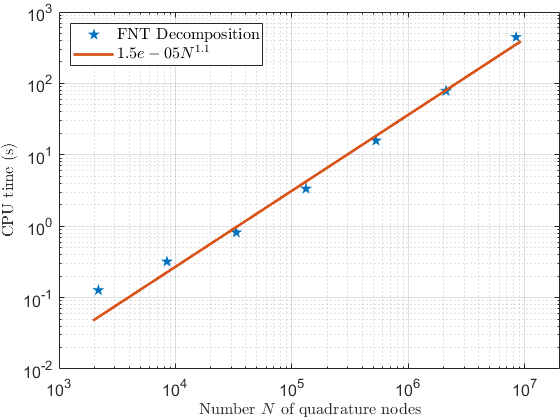}\\
  \subcaption{Decomposition of Multi-Level FaTeNT}\label{fig:time_favest_fwd}
  \end{minipage}\hspace{8mm}
  \begin{minipage}{0.46\textwidth}
  \centering
  \includegraphics[trim = 0mm 0mm 0mm 0mm, width=\textwidth]{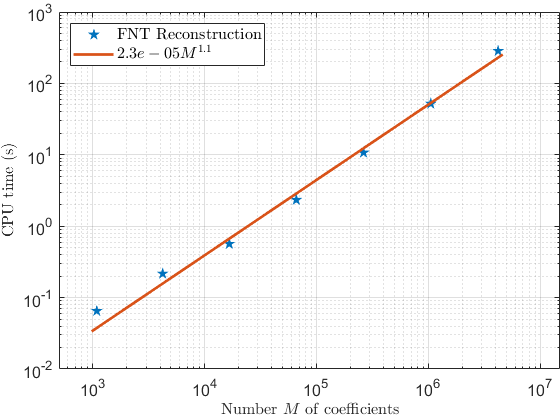}\\
  \subcaption{Reconstruction of Multi-Level FaTeNT}\label{fig:time_faVest_adj}
  \end{minipage}
\end{minipage}
\begin{minipage}{0.9\textwidth}
\vspace{1mm}
\caption{CPU time of decomposition and reconstruction of FaTeNT.}
\label{fig:timecost}
\end{minipage}
\end{minipage}
\end{figure}
%
%
%
\begin{table}[htb]
\centering
\begin{minipage}{0.97\textwidth}
\centering
\scriptsize
\begin{tabular}{l*{8}{c}c}
\toprule
$J$      &  5       &  6  &  7    &   8     &   9    & 10    & 11  \\
\midrule
 $N$           &  2,178    &  8,450 &   33,282  &   132,098   &  526,338 &  2,101,250 & 8,396,802\\

 $t^{\mathrm{dec}}$   & 0.1268 & 0.3196 (2.5) & 0.8175 (2.6)  & 3.3465 (4.1)& 15.7947 (4.7)  & 78.5736 (5.0) & 445.8219 (5.7)\\\midrule

 $M$   & 1,088 &  4,224  & 16,640 &66,048 &  263,168 & 1,050,624  & 4,198,400\\

 $t^{\mathrm{rec}}$  & 0.0650 & 0.2165 (3.3)& 0.5647 (2.6)& 2.3386 (4.1)& 10.7283 (4.6)& 51.9840 (4.8)& 284.9573 (5.5)\\
\bottomrule
\end{tabular}
\end{minipage}
\begin{minipage}{0.9\textwidth}
\vspace{3mm}
\caption{CPU time in seconds in Multi-level FaTeNT decomposition and reconstruction. The $t^{\rm dec}$ is compared with the number of points in decomposion, and $t^{\rm rec}$ is with the number of coefficients in reconstruction. Here $J_0=2$ and $5\leq J\leq11$. FaTeNT decomposition uses Gauss-Legendre tensor rule which has $N=N_J\approx 2^{2J+1}$ nodes and FaTeNT reconstruction uses $M=M_J\approx 2^{2J}$ coefficients. The numbers inside brackets are the ratios $\frac{t^{\mathrm{dec}}(N_{J})}{t^{\mathrm{dec}}(N_{J-1})}$ and $\frac{t^{\mathrm{rec}}(M_{J})}{t^{\mathrm{rec}}(M_{J-1})}$. The numerical test was run under Intel Core i7-4770 CPU @ 3.40GHz with 16GB RAM in Windows 10.}
\label{tab:timecost}
\end{minipage}
\end{table}
Table~\ref{tab:relative-err} reports the relative $l_2$ reconstruction errors $\|T-T^{\rm rec}\|_2/\|T\|_2$ for GL and SD at levels $J=3,4,\ldots,7$ and $J_0=J-1$. We observe that the error depends upon the smoothness of the field, but the impact of choice of quadrature rules (which are polynomial-exact) is negligible.
 We note that the relative $l_2$ errors match those made by SBFs interpolation as reported in \cite{FuWr2009}. For example, when $J=6$, $N=8450$ GL points ($N=8388$ for SD) are used, the reconstruction of FaTeNT have approximation errors of magnitudes $10^{-12}\sim 10^{-10}$, $10^{-4}$, $10^{-3}$ for Tangent Field A, B and C, respectively. That coincides with the results presented in \cite{FuWr2009}, e.g. the mesh-norm $h_X$ is near $0.03$ (when roughly $10,000$ minimum energy (ME) nodes are used) and the corresponding relative errors reach the similar orders, as visualized in Figure~5 in \cite{FuWr2009}. We leave it a future work to probe the dependence of approximation error on the level or number of points used in FaTeNT.
\begin{table}[htb]
\centering
\begin{minipage}{0.97\textwidth}
\centering
\scriptsize
\begin{tabular}{l*{7}{c}c}
\toprule
&$\mbox{Points}$  &   $J=3$    &   $J=4$   &   $J=5$    &  $J=6$     &   $J=7$     \\
\midrule
 \multirow{2}{*}{Vector Field 1}& GL &1.5259e-11 &1.5136e-11   & 1.3203e-11 & 9.4662e-12& 7.7702e-12 \\
                          &SD &  3.3550e-09 & 2.7772e-09 &9.7854e-10 &4.7835e-10  &2.9931e-10  \\

 \midrule
 \multirow{2}{*}{Vector Field 2}& GL  & 1.4251e-01  & 4.2991e-03 &2.5075e-03 & 3.1045e-04 & 5.2139e-05  \\
                          &SD   & 1.3830e-01  &  4.8957e-03 &2.3787e-03 & 3.5217e-04 & 5.8053e-05  \\
 \midrule
 \multirow{2}{*}{Vector Field 3}& GL   &  2.3607e-01 & 4.1204e-02 & 4.5036e-03& 1.1086e-03 & 2.9840e-04 \\
                          &SD   & 2.3971e-01  & 3.1106e-02 &4.6418e-03 & 1.5850e-03 &   3.5757e-04 \\
 \bottomrule
\end{tabular}
\end{minipage}
\begin{minipage}{0.8\textwidth}
\vspace{3mm}
\caption{Relative $l_{2}$ reconstruction errors $\|T-T^{\rm rec}\|_2/\|T\|_2$ with GL and SD quadrature nodes}
\label{tab:relative-err}
\end{minipage}
\end{table}

\subsection{Case study: Earth wind data analysis}
\begin{figure}[h]
\centering
\begin{minipage}{\textwidth}
{\includegraphics[width=0.32\linewidth]{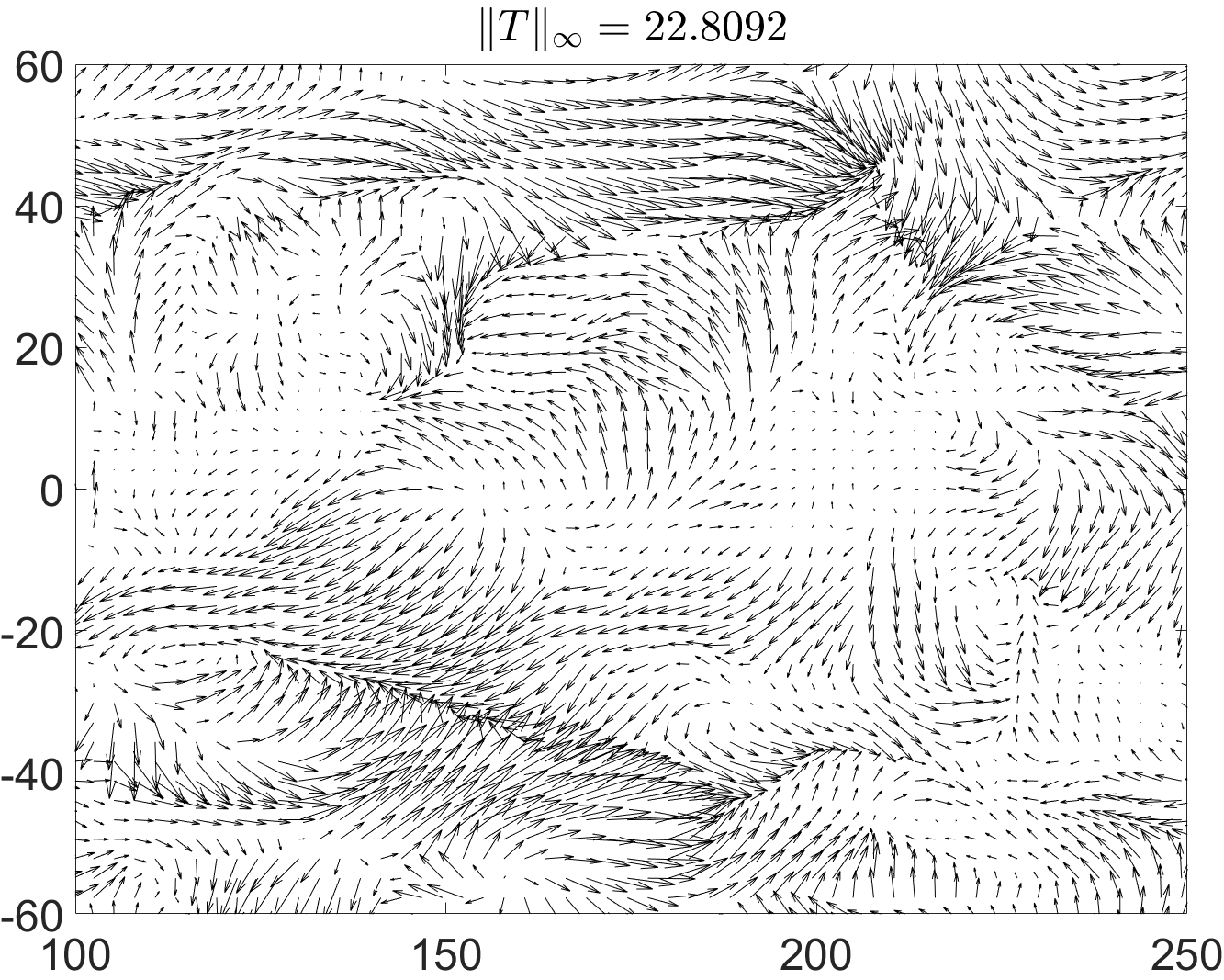}\quad
    \includegraphics[width=0.32\linewidth]{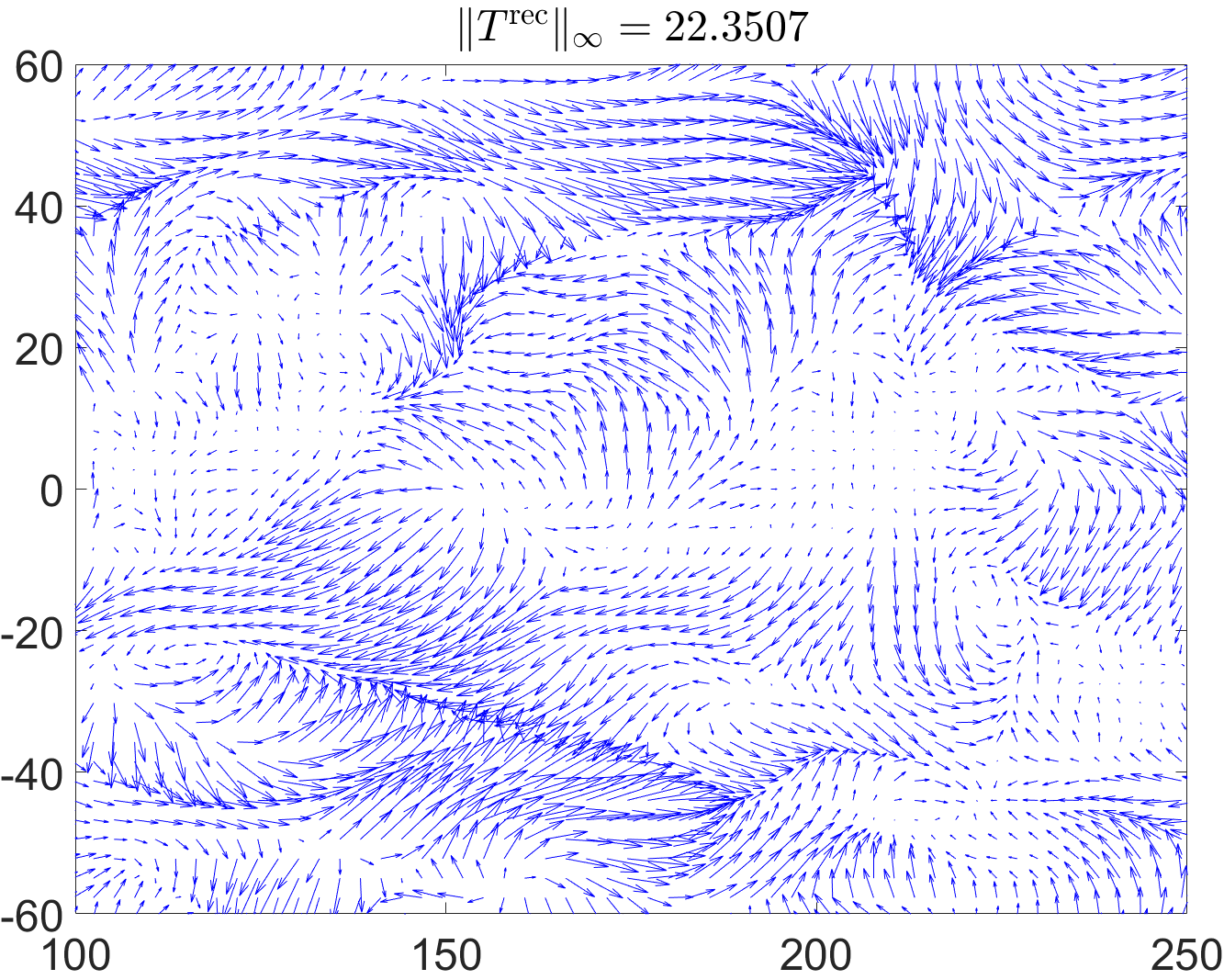}\quad
    \includegraphics[width=0.32\linewidth]{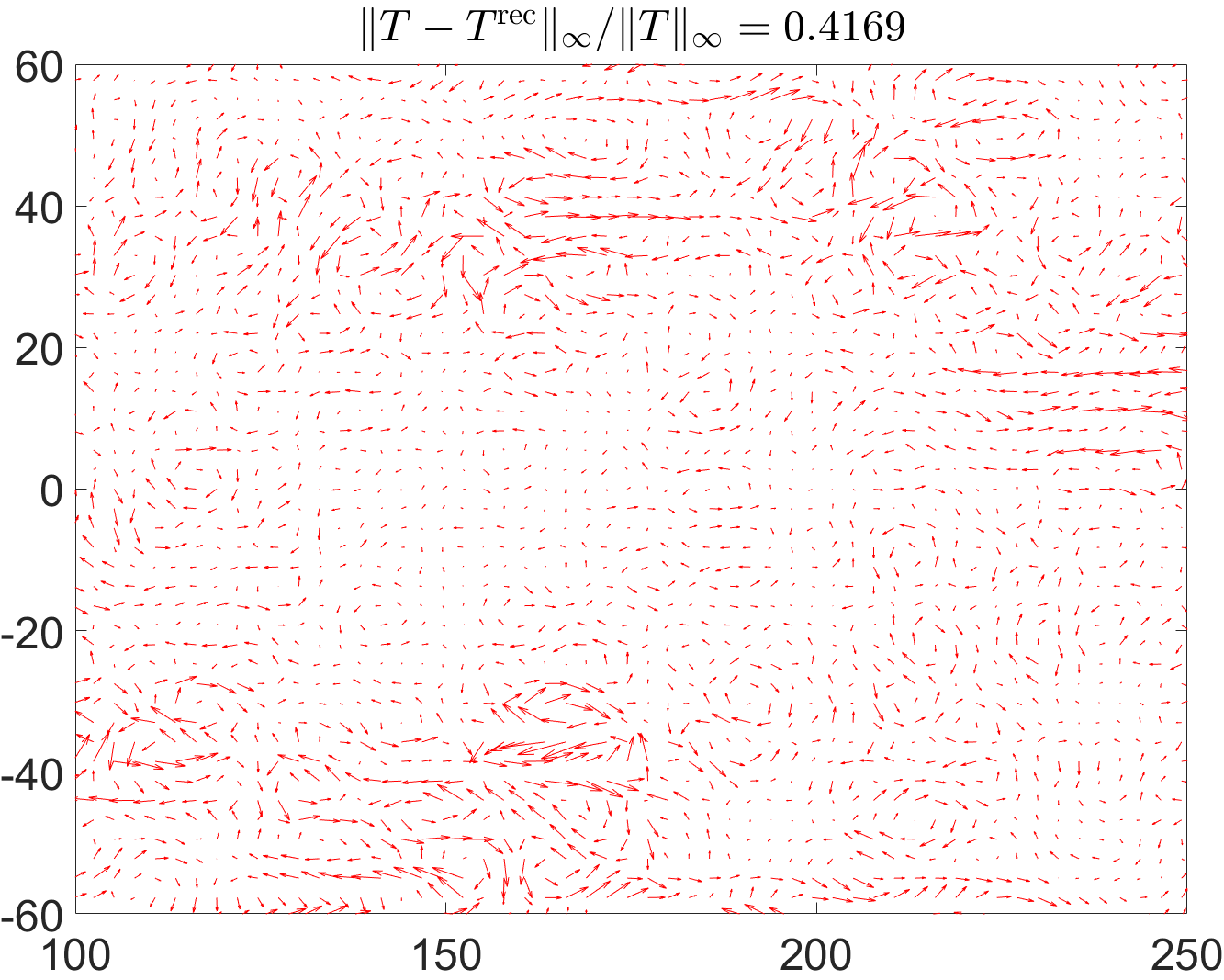}}
\end{minipage}
    \caption{Visualization of earth wind data reconstruction by one level FaTeNT with GL. The first plot shows the real wind field $T$. The second and third plots are the reconstructed field $T^{\rm rec}$ and the point-wise error $T-T^{\rm rec}$. The normalized max norms for $T$, $T^{\rm rec}$ and relative error $(T-T^{\rm rec})/T$ are shown for each case.}\label{casestudy}
\end{figure}
In this subsection, we use a climatological wind field described in NCEP/NCAR Reanalysis \cite{kalnay1996ncep} to test the effectiveness of the proposed FaTeNT algorithms for real world earth wind data. It is common in meteorology to express the wind vector in terms of the following two orthogonal velocity components \cite{kistler2001ncep,kanamitsu2002ncep}.
\begin{itemize}
  \item The $u$ the zonal velocity, which is the component of the horizontal wind towards east.
\item The meridional velocity $v$, which is the component of the horizontal wind towards north.
\end{itemize}
In this case study, we use the NCEP/NCAR Reanalysis daily average data available from the cite of Physical Sciences Division (PSD)\footnote{https://www.esrl.noaa.gov/psd/data/gridded/data.ncep.reanalysis.pressure.html}. In the simulations, we only consider the samples at year 2018 with the first pressure level ($17$ in total) and first time instant ($205$ in total). Its spatial coverage performs in $2.5$ degree $\times 2.5$ degree global grids ($144\times73$).

Based on Equation (16) in \cite{Narcowich2007}, we convert the wind field expressed in $u$ and $v$ to the 3D representation $T=[T_1,T_2,T_3]$. Then we sample the vector fields at $8450$ GL points with $J=6$. It provides an approximation of the real wind field. In Figure~\ref{casestudy}, we plot the real wind field with $8450$ GL samples, the reconstructed vector field by using one level FaTeNT and the residual vector field, from left to right. For a better resolution and visualization, we restrict the range of longitude in $[100,250]$ (the actual range of the original earth wind is $[0,357.5]$) and latitude in $[-60,60]$ (the actual range is $[-90,90]$). The normalized maximum norms, i.e. $\|T\|_{\infty}$,  $\|T^{\rm rec}\|_{\infty}$ and $\|T-T^{\rm rec}\|_{\infty}/\|T\|_{\infty}$, in terms of $8450$ samples, are displayed for each plot. It can be observed that FaTeNT works effectively on reconstruction although the wind field changes rapidly in some region. On the other hand, the relative error $l_2$-error $\|T-T^{\rm rec}\|_{2}/\|T\|_{2}$ is $0.4169$. It should be noted that the error magnitude is mainly caused by the approximation of original scattered instances with GL points, which can be viewed as ``data-preprocessing'' error, and the approximation error of FaTeNT algorithms, which is determined by the smoothness of the real-world vector field.

\section*{Acknowledgements}
This research was partially supported under the Australian Research Council's Discovery Project DP160101366.
We would thank E. J. Fuselier and G. B. Wright for providing their MATLAB codes implementing vector field examples.
The case study for earth wind data analysis is based on NCEP Reanalysis data provided by the NOAA/OAR/ESRL PSD, Boulder, Colorado, USA, from their Web site at \url{https://www.esrl.noaa.gov/psd/}.

\bibliographystyle{apa} 
\bibliography{TensorNeed1}

\end{document}